\documentclass[11pt]{amsart}
\usepackage{latexsym,amssymb,amsmath,graphics}
\usepackage[dvips]{graphicx}
\usepackage{color}

\def\cc{{\mathcal C}}
\def\dd{{\mathcal D}}

\def\hh{{\mathcal H}}

\def\pp{{\mathcal P}}

\def\ffi{\varphi}
\def\eps{\varepsilon}
\def\dst{\displaystyle}

\DeclareMathOperator{\supp}{supp}
\DeclareMathOperator{\vect}{span}
\def\C{{\mathbb{C}}}

\def\N{{\mathbb{N}}}

\def\R{{\mathbb{R}}}

\def\T{{\mathbb{T}}}

\def\Z{{\mathbb{Z}}}

\def\d{\,{\mathrm{d}}}
\newcommand{\norm}[1]{{\left\|{#1}\right\|}}
\newcommand{\ent}[1]{{\left[{#1}\right]}}
\newcommand{\abs}[1]{{\left|{#1}\right|}}
\newcommand{\scal}[1]{{\left\langle{#1}\right\rangle}}

\newcommand{\ps}{\psi_{n,c}^{(\alpha)}}

\newcommand{\wJ}{\widetilde P^{(\alpha)}}
\newcommand{\J}{P^{(\alpha,\beta)}}

\newcommand{\vp}{\psi^{(\alpha)}_{n,c}}
\newcommand{\jc}[1]{j^{(\alpha)}_{#1,c}}

\newtheorem{lemma}{Lemma}[section]
\newtheorem{proposition}[lemma]{Proposition}
\newtheorem{theorem}[lemma]{Theorem}

\theoremstyle{definition}

\theoremstyle{remark}
\newtheorem{remark}[lemma]{Remark}

\begin{document}

\title[$L^p$ convergence of prolates]{Mean convergence of prolate spheroidal series and their extensions}

\author{Mourad Boulsane,  Philippe Jaming \& Ahmed Souabni}

\address{Mourad Boulsane
\noindent Address: University of Carthage,
Department of Mathematics, Faculty of Sciences of Bizerte, Bizerte, Tunisia.}
\email{boulsane.mourad@hotmail.fr}

\address{Philippe Jaming
\noindent Address: Univ. Bordeaux, IMB, UMR 5251, F-33400 Talence, France.
CNRS, IMB, UMR 5251, F-33400 Talence, France.}
\email{Philippe.Jaming@gmail.com}

\address{Ahmed Souabni
\noindent Address: University of  Carthage,
Department of Mathematics, Faculty of Sciences of Bizerte, Bizerte, Tunisia.}
\email{souabniahmed@yahoo.fr}

\begin{abstract}
The aim of this paper is to establish the range of $p$'s for which the expansion of a function $f\in L^p$ in
a generalized prolate spheroidal wave function (PSWFs) basis converges to $f$ in $L^p$.
Two generalizations of PSWFs are considered here, the circular PSWFs introduced by D. Slepian
and the weighted PSWFs introduced by Wang and Zhang.
Both cases cover the classical PSWFs for which the corresponding results has been previously established by Barcel\'o and Cordoba.

To establish those results, we prove a general result that allows to extend mean convergence in a given basis 
(e.g. Jacobi polynomials or Bessel basis) to mean convergence in a second basis (here the generalized PSWFs).
\end{abstract}

\maketitle

\section{Introduction}

In their seminal work from the 70s, Landau, Pollak and Slepian  \cite{LP1,LP2,Slepian1}
have shown that the orthonormal basis that is best concentrated in the time-frequency plane is given by the Prolate Spheroidal Wave Functions (PSFWs). This basis therefore provides an efficient tool for signal processing. Since then, the PSFWs have proven useful in many applications ranging from random matrix theory
({\it e.g.} \cite{dCM,Me,Dy}) to numerical analysis ({\it e.g.} \cite{XRY,Wang}).
While taking naturally place in an $L^2$ setting, one may also consider the behavior of expansions of functions in the
PSFW basis in the $L^p$-setting. This has been done by Barcelo and Cordoba for the usual PSFWs. Our aim here is to extend this work
to two natural generalizations of the PSFWs, namely, the Hankel-PSFWs introduced by Slepian \cite{Slepian2}
and the weighted PSFWs recently introduced by Wang and Zhang \cite{Wang2}.

Let us now be more precise with the results in this paper. First let us recall that the
prolate spheroidal wave functions $(\psi_{n,c})_{n\geq 0}$ are eigenvectors of an integral operator.
Using the min-max theorem, they can thus be obtained as solutions of an extremal problem:
for $c>0$, recall that the Paley-Wiener space $PW_c=\{f\in L^2(\R)\,: \supp \widehat{f}\subset [-c,c]\}$
where $\widehat{f}$ stands for the Fourier transform of $f$.
Then one sets 
$$
\psi_{n,c}=\mathrm{argmax}\left\{\frac{\norm{f}_{L^2(I)}}{\norm{f}_{L^2(\R)}}\,: f\in PW_c,\ 
f\in\vect\{\psi_{k,c},k<n\}^\perp\right\}.
$$
A fundamental fact discovered by Landau, Pollak and Slepian is that
they are also eigenfunctions of a Sturm-Liouville operator, a fact tagged as a ``happy miracle'' by Slepian \cite{Slepian4}. Another key fact four our purpuses is that $(\psi_{n,c})_{n\geq 0}$ is an orthonormal basis of $PW_c$
and this basis is the best concentrated in the time domain.

In this paper, we are interested in two generalizations of the PSFWs.
For both cases, the basis is constructed as a set of eigenvectors of an integral operator, the happy miracle occurs so that they are also eigenvectors of a Sturm-Liouville operator and, more important for us,
they form an orthonormal basis of a Paley-Wiener type of space.

The first basis we consider was introduced by Slepian \cite{Slepian2}.
It is an analogue of the classical PSFWs adapted to higher dimensional {\em radial} Fourier analysis. To introduce them, we need some further notation.
First, we replace the Fourier transform by the {\em Hankel transform} defined for $f\in L^1(0,+\infty)$ by
$$
\mathcal H^{\alpha}f(x)=\int_0^{+\infty}\sqrt{xy}J_{\alpha}(xy) f(y)\d y
$$
where $J_\alpha$ is the Bessel function and $\alpha>-1/2$. Like the usual Fourier transform, the Hankel transform extends into a unitary operator on $L^2(0,+\infty)$.
The corresponding Paley-Wiener space is then denoted by
$$
HB_c^{(\alpha)} =\left\{f\in L^2(0,\infty); \supp \mathcal{H}^{\alpha}(f)\subseteq[0,c]\right\}.
$$
Finally, the Circular (Hankel) Prolate Spheroidal Wave Functions (CPSWFs) are defined by
$$
\psi_{n,c}^\alpha=\mathrm{argmax}\left\{\frac{\norm{f}_{L^2(0,1)}}{\norm{f}_{L^2(0,+\infty)}}\,: f\in HB_c^{(\alpha)},\ 
f\in\vect\{\psi_{k,c}^\alpha,k<n\}^\perp\right\}.
$$
Then $(\psi_{n,c}^\alpha)_{n\geq 0}$ is an orthonormal basis of $HB_c^{(\alpha)}$.
Note also that when $\alpha=0$, these are usual PSFWs, more precisely, $\psi_{n,c}^0=\psi_{2n,c}$.

The second basis we consider, the Weighted Prolate Spheroidal Wawe Functions (WPSFWs), 
is defined in a similar fashion. We first introduce the weighted Paley-Wiener spaces 
$$
wPW_c^{(\alpha)}=\left\{f\in L^2(\R); \supp \widehat{f}\subseteq[-c,c],
\ \widehat{f}\in L^2\bigl((-c,c),(1-x^2/c^2)^{-\alpha}\d x\bigr)
\right\}.
$$
The WPSFWs are defined by
$$
\Psi_{n,c}^\alpha=\mathrm{argmax}\left\{\frac{\norm{f}_{L^2\bigl((-1,1),(1-x^2)^\alpha\d x\bigr)}}
{\norm{\widehat{f}}_{L^2\bigl((-c,c),(1-x^2/c^2)^{-\alpha}\d x\bigr)}}\,: f\in wPW_c^{(\alpha)},\ 
f\in\vect\{\Psi_{k,c}^\alpha,k<n\}^\perp\right\}.
$$
Again, $\Psi_{n,c}^\alpha$ is an orthonormal basis of $wPW_c^{(\alpha)}$ and $\Psi_{n,c}^0=\psi_{n,c}$.

The aim of this paper is to characterize the range of $p$'s for which prolate spheroidal wave 
functions converge in $L^p$.
The subject of the $L^p$-convergence (also called mean convergence of order $p$) 
of orthogonal series, is a central subject in harmonic analysis. 
This kind of convergence is briefly described as follows. 
Let  $1 < p <\infty$ , $a,b \in \overline{\mathbb{R}}$, $I=(a,b)$,
and $\{\phi_n\}$ an orthonormal set of the weighted Hilbert space 
$L^2(I, \omega)$-space, where $\omega$ is a positive weight function. 
We define the kernel 
$$
K_N(x,y)=\sum_{n=0}^N\phi_n(x)\overline{\phi_n(y)}
$$
so that the orthogonal projection of $f\in L^2(I,\omega)$ on the span of $\{\phi_0,\ldots,\phi_N\}$ is given by
$$
\mathcal{K}_N(f)(x)=\int_I K_N(x,y)f(y)\omega(y)\d y=\sum_{n=0}^N a_n(f) \phi_n(x)
$$
with
$$
a_n(f) = \int_I f(y)\overline{\phi_n(y)} \omega(x)\d y.
$$
Now, this last expression may be well defined even for $f\in L^p(I,\omega)$, $p\not=2$ and then
$\mathcal{K}_N(f)$ is also well defined. This happens for instance if $\phi_n\in L^p(I,\omega)$ for every $p$ which is often
the case in practice.
The orthonormal set $\{\phi_n\}$ is said to have mean convergence of order $p$, or $L^p$-convergence 
over the Banach space
$L^p(I,d\omega)$ if for every $f \in L^p(I,d\omega)$, $\mathcal{K}_N(f)$ is well defined and
$$
\lim_{N \to \infty}\norm{f-\mathcal{K}_Nf}_{L^p(I,\omega)} = \lim_{N \to \infty} \bigg[ \int_{a}^{b}|f(x)- \mathcal{K}_N(f)(x) |^p \omega (x)\d x\bigg]^{1/p} =0.$$ 
This concept of mean convergence is also valid on  a subspace $\mathcal B,$ rather than the whole Banach space  $L^p(I,d\omega)$.

To the best of our knowledge, M. Riesz was the first in the late 1920's, to investigate this problem in the special case of the
trigonometric Fourier series over $L^p(\T),$ $1\leq p < +\infty.$ More precisely, in \cite{Riesz}, it has been shown that the 
Hilbert transform over the torus $\tt$ is bounded on $L^p(\T)$ if and only if $ p>1$. Further, the $L^p-$boundedness of the
Hilbert transform is equivalent to the mean convergence of the Fourier series on $L^p(\T)$. 
In the late 1940's, H. Pollard, in a series of papers  \cite{Pollard1,Pollard2,Pollard3}, has studied the mean convergence of some
classical orthogonal polynomials, such as Legendre and Jacobi polynomials. In particular, in the later case, he has shown that if 
$\alpha \geq -\frac{1}{2}$ and $\omega_{\alpha}(x)=(1-x^2)^{\alpha}$, $x\in I=[-1,1]$ is the Jacobi weight,
then the mean convergence over $L^p(I,\omega_{\alpha,})$ of Jacobi series expansion holds true, whenever
$$
m(\alpha):= 4 \frac{\alpha+1}{2\alpha+3} < p <
M(\alpha) := 4 \frac{\alpha+1}{2\alpha+1}.
$$
He has also shown that the previous conclusion fails if $p< m(\alpha)$ or $p>M(\alpha)$. In \cite{Milton}, the authors have 
shown that the mean convergence of the Bessel series expansion over the space $L^p([0,1], x\,\mathrm{d}x)$ holds true whenever $4/3 < p  < 4$.
Later on, Newman and Rudin \cite{Newman} have shown that the mean convergence fails for the critical values of 
$p= m(\alpha), p=M(\alpha)$ in the Jacobi case and for $p=4/3,\,  p=4$ for the Bessel case. 
More recently, in \cite{J.L.VARONA} Varona has extended the mean convergence of Bessel series for $ \alpha > -1/2 $
over the Hankel Paley-Wiener space of functions from $L^p([0,\infty),x^\alpha\,\mathrm{d}x )$ with compactly supported Hankel transforms.  

An other important extension has been given by Barcelo and Cordoba \cite{BC} where they have shown that the series expansion 
in terms of the classical  prolate spheroidal wave functions (PSWFs)  has the mean convergence property over the previous Fourier Paley-
Wiener space, holds true if and only if $ 4/3 < p < 4 $. This is the main source of inspiration for our work, so let us detail
the ideas behind \cite{BC}. Barcelo and Cordoba first determine the expansion of PSFWs in a basis consisting of Bessel functions.
It turns out that the kernel of the projection onto this second basis is given by a Christoffel-Darboux like formula
so that it's mean convergence properties can be deduced from estimates for weighted Hilbert transforms. The last step of the proof
is a sort of transference principle which allows to show that the PSWFs have the mean convergence property of order $p$
exactly when the Bessel basis has this property.

Our first aim here is to formalize this transference principle. We consider two orthonormal bases 
$(\ffi_n)_{n\geq 0}$ and $(\psi_n)_{n\geq 0}$ of $L^2(\Omega,\mu)$. Then, we establish a fairly
general principle giving several conditions on $(\ffi_n)_{n\geq 0}$ and $(\psi_n)_{n\geq 0}$ that will ensure the mutual mean 
convergence property of order $p$ associated for the two bases.

The second part of the paper then consists in applying this principle to the two extensions of PSFWs mentionned above.
For the Circular PSFWs the second basis consists again of a basis built from Bessel functions for which we have to adapt the proof
of Barcelo-Cordoba to establish the range of $p$'s for which mean convergence holds. The case of Weighted PSFWs is a bit simpler
as the second basis consists of Jacobi polynomials for which the mean convergence property is already known. As this case is simpler,
it will be treated first. We may now state our main result:

\medskip

\noindent {\bf Theorem.} {\sl Let $\alpha>-1/2$, $c>0$, $N\geq 0$. Let $I=(-1,1)$ and $\omega_\alpha(x)=(1-x^2)^\alpha$.
\begin{itemize}
	\item Let $p_0=2-\frac{1}{\alpha+3/2}$ and $p_0^{\prime}=\dst 2+\frac{1}{\alpha+1/2}$.
	Let $(\Psi_{n,c}^{(\alpha)})_{n\geq 0}$ be the family of weighted prolate spheroidal wave functions.
	For a smooth function $f$ on $I=(-1,1)$, define
	$$
	\Psi^{(\alpha)}_Nf=\sum_{n=0}^N\scal{f,\Psi_{n,c}^{(\alpha)}}_{L^2(I,\omega_\alpha)}\Psi_{n,c}^{(\alpha)}.
	$$
	Then, for every $p\in(1,\infty)$, $\Psi^{(\alpha)}_N$ extends to a bounded operator $L^p(I,\omega_\alpha(x)\,\mathrm{d}x)\to L^p(I,\omega_\alpha(x)\,\mathrm{d}x)$.
	Further
	$$
	\Psi^{(\alpha)}_Nf\to f\qquad \mbox{in }L^p(I,\omega_\alpha(x)\,\mathrm{d}x)
	$$
	for every $f\in L^p(I,\omega_\alpha(x)\,\mathrm{d}x)$ if and only if $p\in(p_0,p_0^\prime)$.
	\item
	Let $(\psi_{n,c}^{(\alpha)})_{n\geq 0}$ be the family of Hankel prolate spheroidal wave functions.
	For a smooth function $f$ on $I=(0,\infty)$, define
	$$
	\Psi^{(\alpha)}_Nf=\sum_{n=0}^N\scal{f,\psi_{n,c}^{(\alpha)}}_{L^2(0,\infty)}\psi_{n,c}^{(\alpha)}.
	$$
	Then, for every $p\in(1,\infty)$, $\Psi^{(\alpha)}_N$ extends to a bounded operator $L^p(0,\infty)\to L^p(0,\infty)$.
	Further
	$$
	\Psi^{(\alpha)}_Nf\to f\qquad \mbox{in }L^p(0,\infty)
	$$
	for every $f\in B_{c,p}^{\alpha}$ if and only if $p\in(4/3,4)$.
\end{itemize}
}

\medskip

This work is organized a follows. In section 2, we study a general principle that ensure the mutual $L^p$-convergence of two series expansion with respect to two different orthonormal bases of a Hilbert space $L^2(\mu) $. In section 3, we give a list of technical lemmas that ensure or simplify the conditions given in the general principle of the previous section. In section 4, we apply the results of sections 2 and 3 and check in detail that the conditions that we have established in the case of general principle hold true for the series expansion in the weighted PSWFs. Finally in section 5, we prove that this mean convergence property holds also true for circular PSWFs series.

\section{The General principal}

\subsection{The setting and the main result}
\label{sec:mainth}

As already explained, to prove the $L^p$-convergence of the expansion in a prolate basis, we will expend the
prolates in a second basis for which this $L^p$-convergence is easier to study. This idea is formalized in the following setting:

We consider a measure space $(\Omega,\mu)$ and assume that, for every $1<p<\infty$,
$L^p(\Omega,\mu)$ is infinite dimensional and separable. The dual index of $p$ will be denoted by $p'=\dst\frac{p}{p-1}$.
We consider two orthonormal bases $(\ffi_n)_{n\geq 0}$ and $(\psi_n)_{n\geq 0}$ of $L^2(\Omega,\mu)$. For $N\geq 0$
we denote by $\Phi_N$ (resp. $\Psi_N$) both the orthogonal projection on $\vect\{\ffi_0,\ldots,\ffi_N\}$
(resp. $\vect\{\psi_0,\ldots,\psi_N\}$) and its kernel
$$
\Phi_N(x,y)=\sum_{n=0}^N\ffi_n(x)\overline{\ffi_n(y)}
\quad\mbox{resp.}\quad
\Psi_N(x,y)=\sum_{n=0}^N\psi_n(x)\overline{\psi_n(y)}.
$$

Our aim in this section is to define several conditions on $\ffi_n,\psi_n$ that will ensure that, for any $1<p<\infty$,
$\Phi_Nf\to f$ in some $L^p$ if and only if $\Psi_Nf\to f$ in $L^p$.
The first condition is of course that this makes sense. The second one is that some relation exists between
the two bases. The other conditions are technical and are those that will be the most difficult to check in practice.

\begin{itemize}
\item[$(L)$] For every $1<p<\infty$, and every $n$, $\ffi_n\in L^p(\mu)$.
Further, we assume that there is a $0\leq \gamma_p<1$ such that 
\begin{equation}
\label{eq:lpind}
\norm{\ffi_n}_{L^p(\mu)}\lesssim n^{\gamma_p}.
\end{equation}
Finally, we assume that $0<\alpha_p:=\gamma_p+\gamma_{p'}<1$ and that there is a $p_0$ such that
if $p\in (p_0,p_0^{\prime})$, $\alpha_p=0$. In other words, for $p\in (p_0,p_0^{\prime})$,
\begin{equation}
\label{eq:lp1}
\norm{\ffi_n}_{L^p(\mu)}\norm{\ffi_n}_{L^{p'}(\mu)}\lesssim 1
\end{equation}
while for $p\notin (p_0,p_0^{\prime})$,
\begin{equation}
\label{eq:lp2}
\norm{\ffi_n}_{L^p(\mu)}\norm{\ffi_n}_{L^{p'}(\mu)}\lesssim n^{\alpha_p},\qquad \alpha_p<1.
\end{equation}

\item[$(R)$] Let $\alpha_k^n=\scal{\psi_n,\ffi_k}_{L^2(\mu)}$ so that $\psi_n=\sum_{n=0}^\infty \alpha_k^n\ffi_k$. We assume that
there exists an integer $n_0$ and $\kappa,\kappa'>0$ two real numbers such that
$(\alpha_k^n)$ satisfies a three term recursion formula
$$
f(k,n)\alpha_k^n=a_k\alpha_{k-1}^n+a_{k+1}\alpha_{k+1}^n
$$
where
\begin{enumerate}
\item $|a_k|\leq\frac{1}{2}$,

\item for fixed $n$, there is a $k_n$ such that $|f(k,n)|\gtrsim k^2$ when $k\geq k_n$,

\item there is an $n_0\geq0$ such that, for $n\geq n_0$, and every $k\geq 0$, $|f(k,n)|\gtrsim k|k-n|$,

\item
\begin{equation}
\label{eq:diff}
\abs{\frac{a_{n+1}}{f(n+1,n)}-\frac{a_{n+2}}{f(n+2,n+1)}}\lesssim n^{-2}.
\end{equation}
\end{enumerate}

\item[$(B)$] Let $\tilde\Phi_N(x,y)=\dst\sum_{n=0}^N\ffi_n(x)\overline{\ffi_{n+1}(y)}$ and write also
$\tilde\Phi_N$ for the corresponding integral operator. For every $1<p<\infty$, we assume that $\tilde\Phi_N$
defines a bounded linear operator on $L^p(\mu)$ and that there exists $\beta_p<1$ such that,
for every $f\in L^p(\mu)$
$$
\norm{\tilde\Phi_N f}_{L^p(\mu)}\lesssim N^{\beta_p}\norm{f}_{L^p(\mu)}.
$$

\item[(C)] There exists $1<p_0<2$ such that $\Phi_Nf\to f$
for every $f\in L^p(\Omega,\mu)$, with convergence in $L^p(\Omega,\mu)$, if and only if $p_0<p<p_0^\prime$.

\item[(D)] There exists a set $\dd$ that is dense in every $L^p(\mu)$, $1<p<\infty$, such that,
for every $1<p<\infty$ and every $f\in L^p(\mu)$, $\Phi_Nf,\Psi_Nf\to f$ in $L^p(\mu)$ when $N\to\infty$.
\end{itemize}

In this all of Section \ref{sec:mainth} we will use the above notation and assume that these conditions are fulfilled.
Our main result is then:

\begin{theorem}\label{th:main}
With the above notation, and under conditions $(L)$, $(R)$, $(B)$, $(C)$ and $(D)$, we have
$\Psi_Nf\to f$
for every $f\in L^p(\Omega,\mu)$, with convergence in $L^p(\Omega,\mu)$, if and only if $p_0<p<p_0^\prime$.
\end{theorem}

\begin{remark}
\label{re:thmain}
Note that the adjoint $\tilde\Phi_N^*$ of $\tilde\Phi_N$ has kernel 
$\tilde\Phi_N^*(x,y)=\dst\sum_{n=1}^{N+1}\ffi_n(x)\overline{\ffi_{n-1}(y)}$.
Thus, if condition $(B)$ holds, then for every $f\in L^p(\mu)$
$$
\norm{\tilde\Phi_N^* f}_{L^p(\mu)}\lesssim N^{\beta_{p'}}\norm{f}_{L^p(\mu)}.
$$
Condition $(B)$ may be replaced by a slightly weaker condition, see Remark \ref{rem:condB} below.

Also we state the various conditions with $1<p<\infty$. It is enough that they hold for
$p_1<p<p_1^{\prime}$ with $1<p_1<p_0$.
\end{remark}

The remaining of this section is devoted to the proof of this result.

\subsection{Step 1: A simple lemma and an extension of the Banach-Steinhaus Theorem}

We will here formalize a result that has already been used in \cite{BC}.
To start, let us state the following simple and well known lemma that we prove for sake of completeness:

\begin{lemma}
\label{lem:triv}
Let $1<p<\infty$ and let $K\,:\Omega\times\Omega\to\C$ be such that
$$
\norm{K}_{L^p(\mu)\otimes L^{p'}(\mu)}:=\left(\int_\Omega\left(\int_\Omega|K(x,y)|^{p'}\,\mathrm{d}\mu(y)\right)^{p/p'}
\mathrm{d}\mu(x)\right)^{1/p}<+\infty.
$$
Then the integral operator $K$ defined by
$$
Kf(x)=\int_\Omega K(x,y)f(y)\,\mathrm{d}y
$$
extends to a continuous operator $K\,: L^p\to L^p$ with norm 
$$
\norm{K}_{L^p(\mu)\to L^p(\mu)}\leq\norm{K}_{L^p(\mu)\otimes L^{p'}(\mu)}.
$$
\end{lemma}

\begin{proof}
Indeed, using H\"older's inequality,
\begin{eqnarray*}
\norm{Kf}_p^p&=&\int_\Omega\abs{\int_\Omega K(x,y)f(y)\,\mathrm{d}\mu(y)}^p\,\mathrm{d}\mu(x)\\
&\leq&\int_\Omega\left(\int_\Omega|K(x,y)|^{p'}\,\mathrm{d}\mu(y)\right)^{p/p'}\int_\Omega|f(y)|^{p}\,\mathrm{d}\mu(y)\,\mbox{d}\mu(x)\\
&=&\norm{K}_{L^p\bigl(\mu)\otimes L^{p'}(\mu)}^p\norm{f}_{L^p(\mu)}^p
\end{eqnarray*}
as claimed.
\end{proof}

With condition $(L)$ we can now make sense of $\Phi_Nf,\Psi_Nf$ for every $f\in L^p(\mu)$. Moreover,
according to the Banach-Steinhaus Theorem, the following are equivalent:
\begin{enumerate}
\renewcommand{\theenumi}{\roman{enumi}}
\item for every $f\in L^p(\mu)$, $\Phi_Nf\to f$ 
in $L^p(\mu)$;

\item there exists a dense set $\dd\subset L^p(\mu)$ such that, for every $f\in\dd$,
$\Phi_Nf\to f$ 
in $L^p(\mu)$ and for every $f\in L^p(\mu)$, $\|\Phi_Nf\|_{L^p(\mu)}\lesssim\|f\|_{L^p(\mu)}$.
\end{enumerate}

The statement hold of course with $\Phi_N$ replaced by $\Psi_N$.

Now, as we assume conditions $(D)$, $L^p$-convergence of $\Phi_Nf\to f$, $\Psi_Nf\to f$, is equivalent to the uniform boundedness of
$\norm{\Phi_N}_{L^p(\mu)\to L^p(\mu)}$, and $\norm{\Psi_N}_{L^p(\mu)\to L^p(\mu)}$.
But then, under condition $(C)$, the uniform boundedness of $\norm{\Phi_N}_{L^p(\mu)\to L^p(\mu)}$ holds if and only if
$p_0<p<p_0^{\prime}$. But if $\norm{\Phi_N-\Psi_N}_{L^p(\mu)\to L^p(\mu)}$ is uniformly bounded for every $p$, then
we also get that $\norm{\Psi_N}_{L^p(\mu)\to L^p(\mu)}$ is  uniformly bounded if and only if
$p_0<p<p_0^{\prime}$. We may summarize this discussion in the following proposition:

\begin{proposition}
With the notation of Section \ref{sec:mainth} and under conditions $(L)$, $(R)$, $(B)$, $(C)$ and $(D)$, the following are equivalent:
\begin{enumerate}
\renewcommand{\theenumi}{\roman{enumi}}
\item $\Psi_Nf\to f$
for every $f\in L^p(\Omega,\mu)$, with convergence in $L^p(\Omega,\mu)$, if and only if $p_0<p<p_0^\prime$.

\item for every $1<p<\infty$, there exists a constant $C$ such that, for every $N\geq 0$ and every $f\in L^p(\mu)$,
$$
\norm{\Phi_Nf-\Psi_Nf}_{L^p(\mu)}\leq C\norm{f}_{L^p(\mu)}.
$$
\end{enumerate}
\end{proposition}

\subsection{Step 2: The behavior of the sequence $\alpha_k^n$}

\begin{lemma}
\label{lem:rec}
Let $n$ be an integer, $(a_k)$, $(f_k)$ be two sequences such that $a_1\not=0$, $|f_1|\lesssim  n^2$
for every $k$, $\dst|a_k|\leq\frac{1}{2}$, $|f_k|\gtrsim k|k-n|$.
Let
$(\alpha_k)_{k\geq0}$ be a sequence such that

-- $\dst\sum_{k=0}^\infty |\alpha_k|^2=1$;

-- $(\alpha_k)_{k\geq0}$ satisfies a three term
recursion formula
$$
f_k\alpha_k=a_k\alpha_{k-1}+a_{k+1}\alpha_{k+1}.
$$
Then there exists $\kappa,n_1$ depending only on the constants appearing in the above $\lesssim$ and $\gtrsim$ inequalities
such that, if $n\geq n_1$,
\begin{enumerate}
\renewcommand{\theenumi}{\roman{enumi}}
\item $|\alpha_0|\lesssim (\kappa n)^{3-n}$,

\item for $k\geq 1$, $|\alpha_k|\lesssim (Cn)^{-|k-n|}$

\item $|\alpha_n|^2=1-\eta$ with  $0<\eta\lesssim n^{-2}$.
\end{enumerate}
\end{lemma}


\begin{proof}
First, as $\sum|\alpha_k|^2=1$, $|\alpha_k|\leq 1$ for every $k\geq 0$.

We will first prove the estimate for $k\geq 1$ and write $|f_k|\geq \kappa' k|n-k|\geq \kappa  n$, $\kappa =\kappa'/4$.
As $|a_k|\leq\frac{1}{2}$ for every $k$, we get $|f_k||\alpha_k|\leq 1$ thus
$|\alpha_k|\leq (\kappa n)^{-1}$ if $|k-n|\geq 1$.
Assume now that we have proven that, for $J\geq 1$ we have proven that, for every
$k\geq 1$,
$$
|\alpha_k|\leq (\kappa n)^{-\min(|k-n|,J)}.
$$
Then, if $|k-n|\geq J+1$,
$$
|f_k||\alpha_k|\leq (\kappa n)^{-J}.
$$
As $|f_k|\geq \kappa n$, we obtain
$$
|\alpha_k|\leq (\kappa n)^{-(J+1)}
$$
as claimed. This induction does not allow to estimate $\alpha_0$ for which we instead use the induction
formula in a rougher way: we assumed that there is a constant $\tilde C$ such that $|f_1|\leq \tilde Cn^2$
$$
|a_1||\alpha_0|\leq |f_1||\alpha_1|+\frac{1}{2}|\alpha_2|\leq \tilde C n^2 (\kappa n)^{1-n}
+\frac{1}{2}(\kappa n)^{2-n}\leq \left(\frac{\tilde C}{\kappa^2}+\frac{1}{2\kappa n}\right) (\kappa n)^{3-n}.
$$
from a bound of the form $|\alpha_0|\geq \kappa (\kappa n)^{3-n}$ follows.

Finally, if $n>\max(4,1/2\kappa)$
\begin{eqnarray*}
|\alpha_n|^2&=&1-\sum_{|k-n|\geq 1}|\alpha_k|^2
\geq 1-\kappa^2 (\kappa n)^{6-2n}-2\sum_{j\geq 1}(\kappa n)^{-2j}\\
&\geq&1-\kappa^2 (\kappa n)^{-2}-2(\kappa n)^{-2}\bigl(1-(\kappa n)^{-2}\bigr)\\
&\geq&1-(\kappa^2+4)(\kappa n)^{-2}
\end{eqnarray*}
as claimed.
\end{proof}

Let us now state what this lemma implies on $(\alpha_k^n)$ satisfying condition $(R)$.

According to Lemma \ref{lem:rec}, and up to replacing eventually $n_0$ by $\max(n_0,n_1)$,
we may assume that, if $n\geq n_0$,
\begin{enumerate}
\renewcommand{\theenumi}{\roman{enumi}}
\item for every $n$, $|\alpha_k^n|\lesssim k^{-2}$ (with a constant that depends on $n$);
\item if $n\geq n_0$,
\begin{enumerate}
\item $|\alpha_0^n|\lesssim n^{-2}$,

\item for $k\geq 1$, $|\alpha_k^n|\lesssim (\kappa n)^{-|k-n|}$

\item $|\alpha_n^n|^2=1-\eta_n$ with  $0<\eta_n\lesssim n^{-2}$.
\end{enumerate}
\end{enumerate}

Let us show that this implies that $(\psi_n)$ also satisfies condition $(L)$:

\begin{lemma}\label{lem:lppsi}
With the notation of Section \ref{sec:mainth} and under conditions $(L)$, $(R)$, the sequence $(\psi_n)_{n\geq 0}$ also satisfies condition $(L)$.
\end{lemma}

\begin{proof}[Proof of Lemma \ref{lem:lppsi}]
We write $\dst\psi_n=\sum_{k=0}^\infty \alpha_k^n\ffi_k$ so that
$$
\norm{\psi_n}_{L^p(\mu)}\leq \sum_{k=0}^\infty |\alpha_k^n|\norm{\ffi_k}_{L^p(\mu)}
\lesssim\sum_{k=0}^{\infty} (1+k)^{-2+\gamma_p}<+\infty.
$$
Further, if $n\geq n_0$,
\begin{eqnarray*}
\norm{\psi_n}_{L^p(\mu)}
&\leq&|\alpha_0^n|\norm{\ffi_0}_{L^p(\mu)}+\sum_{k=1}^{n-1}|\alpha_k^n|\norm{\ffi_k}_{L^p(\mu)}+\norm{\ffi_n}_{L^p(\mu)}
+\sum_{k=n+1}^{\infty}|\alpha_k^n|\norm{\ffi_k}_{L^p(\mu)}\\
&\lesssim&\left(n^{-2}+2\sum_{k=1}^{\infty}(\kappa n)^{-|k-n|}+1\right)n^{\gamma_p}\lesssim n^{\gamma_p}
\end{eqnarray*}
as claimed.
\end{proof}

\subsection{Step 3: The decomposition of $\Psi_N$}

In order to prove the theorem, we need to decompose $\Psi_N$ in the basis $(\ffi_n)_{n\geq 0}$.

Recall that $\psi_n=\sum_{k=0}^\infty \alpha_k^n\ffi_n$ and that $(\alpha_k^n)$ satisfy $(R)$.

The decomposition of $\Psi_N$ is the following:
\begin{eqnarray}
\Psi_N(x,y)&=&\sum_{n=0}^N\psi_n(x)\overline{\psi_n(y)}
=\sum_{n=0}^{n_0}\psi_n(x)\overline{\psi_n(y)}+
\sum_{n=n_0+1}^N\sum_{k=0}^\infty\sum_{\ell=0}^\infty\alpha_k^n\alpha_{\ell}^n\ffi_k(x)\overline{\ffi_\ell(y)}\nonumber\\
&=&\Phi_N(x,y)+K_1(x,y)-K_2(x,y)-K_3(x,y)+K_4(x,y)+K_5(x,y)+K_6(x,y)
\label{decomgen} 
\end{eqnarray}
with
$$
K_1(x,y)=\sum_{n=0}^{n_0}\psi_n(x)\overline{\psi_n(y)}
\quad,\quad K_2(x,y)=\sum_{n=0}^{n_0}\ffi_n(x)\overline{\ffi_n(y)},
$$
$$
K_3(x,y)=\sum_{n=n_0+1}^N\eta_n\ffi_n(x)\overline{\ffi_n(y)},
$$
$$
K_4(x,y)=\sum_{n=n_0+1}^N\alpha_n^n\overline{\alpha_{n+1}^n}\ffi_n(x)\overline{\ffi_{n+1}(y)}
\quad,\quad K_5(x,y)=\sum_{n=n_0+1}^N\alpha_{n+1}^n\overline{\alpha_n^n}\ffi_{n+1}(x)\overline{\ffi_n(y)}
$$
and
$$
K_6(x,y)=\sum_{n=n_0+1}^N\sum_{|k-n|\geq 2}\sum_{|\ell-n|\geq 2}\alpha_k^n\alpha_{\ell}^n\ffi_k(x)\overline{\ffi_\ell(y)}.
$$
Let us write $K_jf(x)=\dst\int_\Omega K_j(x,y)f(y)\,\mbox{d}\mu(y)$ for the corresponding integral operators.
We want to bound $\norm{K_j}_{L^p(\mu)\to L^p(\mu)}$ independently from $N$.

According to Lemma \ref{lem:triv}
$$
\norm{K_1}_{L^p(\mu)\to L^p(\mu)}\leq C_1:=\sum_{n=0}^{n_0}\norm{\psi_n}_{L^p(\mu)}\norm{\psi_n}_{L^{p'}(\mu)}
$$
while
$$
\norm{K_2}_{L^p(\mu)\to L^p(\mu)}\leq C_2:=\sum_{n=0}^{n_0}\norm{\ffi_n}_{L^p(\mu)}\norm{\ffi_n}_{L^{p'}(\mu)}
$$
and these two quantities are finite according to condition $(L)$ and do not depend on $N$.

Further, thanks again to condition $(L)$,
$$
\norm{K_3}_{L^p\to L^p}\leq \sum_{n=n_0+1}^N|\eta_n|\norm{\ffi_n}_{L^p(\mu)}\norm{\ffi_n}_{L^{p'}(\mu)}
\lesssim\sum_{n=1}^{\infty}\frac{1}{n^{2-\alpha_p}}<+\infty.
$$

For $K_6$ we will use Lemma \ref{lem:rec} to see that, if $|k-n|\geq 2$, $|\alpha_k^n|\lesssim (\kappa n)^{-|k-n|}$ if $k\not=0$.
On the other hand $\norm{\ffi_k}_{L^p(\mu)}\lesssim k^{\gamma_p}$ with $\gamma_p<1$.
Note also that if we denote by $S_j=\dst\sum_{k=j}^\infty (\kappa n)^{-k}$ then $S_j\lesssim (\kappa n)^{-j}$. We then have to estimate
\begin{eqnarray*}
\sum_{k\geq n+2}|\alpha_k^n|\norm{\ffi_k}_{L^p(\mu)}&\lesssim& \sum_{j\geq 2}(n+j)^{\gamma_p}(\kappa n)^{-j}
=\sum_{j\geq 2}(n+j)^{\gamma_p}(S_j-S_{j+1})\\
&=&(n+2)^{\gamma_p}S_2+\sum_{j=3}^\infty\bigl((n+j)^{\gamma_p}-(n+j-1)^{\gamma_p}\bigl)S_j
\lesssim n^{\gamma_p-2}
\end{eqnarray*}
and as $|\alpha_0^n|\lesssim n^{-2}$,
$$
\sum_{0\leq k\leq n-2}|\alpha_k^n|\norm{\ffi_k}_{L^p(\mu)}\lesssim  n^{-2}+n^{\gamma_p}\sum_{j\geq 2}(\kappa n)^{-j}\lesssim n^{\gamma_p-2}.
$$
It follows that
\begin{eqnarray*}
\norm{K_6}_{L^p(\mu)\to L^p(\mu)}&\leq&
\sum_{n=n_0}^{\infty}\sum_{|k-n|\geq 2}\sum_{|\ell-n|\geq 2}|\alpha_k^n||\alpha_{\ell}^n|\norm{\ffi_k}_{L^p(\mu)}\norm{\ffi_\ell}_{L^{p'}(\mu)}\\
&\lesssim&\sum_{n=n_0}^{\infty}\sum_{|k-n|\geq 2}k^{\gamma_p}|\alpha_k^n|\sum_{|\ell-n|\geq 2}\ell^{\gamma_{p'}}|\alpha_{\ell}^n|\\
&\lesssim&\sum_{n=n_0+1}^{\infty}n^{-4+\alpha_p}<+\infty
\end{eqnarray*}
since $\alpha_p<1$.

The terms $K_4$ and $K_5$ are the most difficult to treat. As they are similar, we will only show $L^p$-boundedness of the first one.
To start, we use $(R)$ to rewrite
\begin{multline*}
K_4(x,y)=\sum_{n=n_0+1}^N\alpha_n^n\overline{\alpha_{n+1}^n}\ffi_n(x)\overline{\ffi_{n+1}(y)}\\
=\sum_{n=n_0+1}^N\frac{\overline{a_{n+1}}}{f(n+1,n)}|\alpha_n^n|^2\ffi_n(x)\overline{\ffi_{n+1}(y)}
+\sum_{n=n_0+1}^N\frac{\overline{a_{n+2}}}{f(n+1,n)}\alpha_n^n\overline{\alpha_{n+2}^n}\ffi_n(x)\overline{\ffi_{n+1}(y)}\\
=K_4^1(x,y)+K_4^2(x,y).
\end{multline*}
Now
\begin{eqnarray*}
\norm{K_4^2}_{L^p\bigl(\mu)\otimes L^{p'}(\mu)}^p&\lesssim&
\sum_{n=n_0+1}^N\frac{|a_{n+2}|}{|f(n+1,n)|}|\alpha_n^n||\alpha_{n+2}^n|\norm{\ffi_n}_{L^p(\mu)}\norm{\ffi_n}_{L^{p'}(\mu)}\\
&\lesssim& \sum_{n=n_0+1}^\infty n^{-3+\alpha_p}<+\infty.
\end{eqnarray*}
since $|a_{n+2}|\lesssim1$, $|f(n+1,n)|\gtrsim n$, $|\alpha_n^n|\leq 1$, $|\alpha_{n+2}^n|\lesssim n^{-2}$
and Property $(L^p)$.

Next, writing $\tilde\alpha_n=\dst\frac{\overline{a_{n+1}}}{f(n+1,n)}|\alpha_n^n|^2$ and using Abel summation,
we get
$$
K_4^1(x,y)=\sum_{n=n_0}^{N-1}\bigl(\tilde\alpha_n-\tilde\alpha_{n+1}\bigr)\tilde\Phi_n(x,y)
+\tilde\alpha_N\tilde\Phi_{N-1}(x,y)
$$

Note that $|\tilde\alpha_n|\lesssim n^{-1}$ so that, with $(B)$, $\norm{\tilde\alpha_N\tilde\Phi_{N-1}}_{L^p(\mu)\to L^p(\mu)}\lesssim 1$.
Further $|\alpha_{n+1}^{n+1}|^2=|\alpha_n^n|^2+\eta_n-\eta_{n+1}$ thus
$$
\tilde\alpha_{n+1}=\frac{\overline{a_{n+2}}}{f(n+2,n+1)}|\alpha_n^n|^2+O(n^{-3})
$$
since $|\eta_n-\eta_{n+1}|\lesssim n^{-2}$, $|a_{n+2}|\lesssim1$, $|f(n+2,n+1)|\gtrsim n$. Thus, using
\eqref{eq:diff} we get $|\tilde\alpha_n-\tilde\alpha_{n+1}|\lesssim n^{-2}$.
It follows that
$$
\norm{\sum_{n=n_0}^{N-1}\bigl(\tilde\alpha_n-\tilde\alpha_{n+1}\bigr)\tilde\Phi_n}_{L^p(\mu)\to L^p(\mu)}
\lesssim \sum_{n=n_0}^{N-1} n^{-2}\norm{\tilde\Phi_n}_{L^p(\mu)\to L^p(\mu)}
\lesssim \sum_{n=n_0}^\infty n^{-2+\beta_p}<+\infty.
$$
This shows that $K_4^1$ is also a bounded operator $L^p(\mu)\to L^p(\mu)$ with bound independent on $N$.
The proof for $K_5$ being similar, we conclude that each term in \eqref{decomgen} defines a bounded operator $L^p(\mu)\to L^p(\mu)$
with bound independent on $N$ and the proof of the theorem is complete.
\hfill$\Box$

\begin{remark}
\label{rem:condB}
By treating simultaneously the terms $K_4$ and $K_5$, it is enough to assume the following slightly weaker condition:
\begin{enumerate}
\item[$(B')$]
Let $\hat\Phi_N(x,y)=\dst\sum_{n=0}^N\ffi_n(x)\overline{\ffi_{n+1}(y)}+\ffi_{n+1}(x)\overline{\ffi_n(y)}$ and write also
$\hat\Phi_N$ for the corresponding integral operator. For every $1<p<\infty$, we assume that $\hat\Phi_N$
defines a bounded linear operator on $L^p(\mu)$ and that there exists $\beta_p<1$ such that,
for every $f\in L^p(\mu)$
$$
\norm{\hat\Phi_N f}_{L^p(\mu)}\lesssim N^{\beta_p}\norm{f}_{L^p(\mu)}.
$$
\end{enumerate}
\end{remark}

\section{Preliminaries and technical Lemmas}

In this section, we will gather some facts from the literature and some simple technical lemmas that will allow to easily establish the conditions of Theorem \ref{th:main}.

\subsection{Condition $(R)$}

\begin{lemma}
\label{lem:fkn}
Let $a,b,c,d,e,\ell\in\R$. Let $(e_{k,n})_{k,n\in\N}$ be a bounded sequence with $|e_{k,n}|\leq e$.
Let
$$
f(k,n)=(an+b)(cn+d)-(ak+b)(ck+d)+e_{k,n}.
$$
Let $(a_n)$ be a sequence such that $a_n=\ell+\tilde a_n$ with $|\tilde a_n|\lesssim n^{-1}$.
Then there exists $n_0$ such that
\begin{enumerate}
\renewcommand{\theenumi}{\roman{enumi}}
\item for fixed n, there exists $k_n$ such that, if $k\geq k_n$, $|f(k,n)|\geq\dst \frac{ac}{2}k^2$;

\item if $n\geq n_0$, $|f(k,n)|\geq\dst\frac{ac}{2}n|n-k|$;

\item if $n\geq n_0$, $\dst\abs{\frac{a_{n+1}}{f(n+1,n)}-\frac{a_{n+2}}{f(n+2,n+1)}}\lesssim n^{-2}$.
\end{enumerate}
\end{lemma}

\begin{proof} Up to replacing $f$ by $-f$ we may assume that $ac>0$.
The first part is trivial as, for fixed $n$, $\dst k^{-2}f(k,n)\to ac$ when $k\to\infty$.

For the second part, the result is trivial for $k=n$ so let us first consider the case $k>n$ and write $k=n+p$, $p\geq 1$.
Now
\begin{eqnarray*}
f(n+p,n)&=&-acp^2-p\bigl(2acn+cb+ad\bigr)+e_{n+p,n}\\
&=&-\frac{ac}{2}pn -p\left(ac(n+p)+\frac{ac}{4}n+cb+ad\right)-\left(\frac{ac}{4}pn-e_{n+p,n}\right).
\end{eqnarray*}
Now $ac(n+p)\geq 0$, $\dst\frac{ac}{4}n+cb+ad\geq 0$ if $n\geq -\frac{4(cb+ad)}{ac}$
and $\frac{ac}{4}pn-e_{n+p,n}\geq \frac{ac}{4}n-e\geq 0$ if $n\geq\frac{4e}{ac}$.
It follows that, if $n$ is large enough $f(n+p,n)\leq\dst-\frac{ac}{2}pn$.

Let us now turn to the case $0\leq k< n$ and write
$k=n-p$ with $1\leq p\leq n$. Then
\begin{eqnarray*}
f(n-p,n)&=&-acp^2+p\bigl(2acn+cb+ad\bigr)+e_{n-p,n}\\
&\geq&\frac{ac}{2}np+p\left(ac(n-p)+\frac{ac}{4}n+cb+ad\right)+ \frac{ac}{4}np+e_{n-p,n}.
\end{eqnarray*}
Now, $n\geq p$ thus $ac(n-p)\geq 0$, and the two other terms are treated as previously.

For the last assertion, first write
$$
\frac{a_{n+1}}{f(n+1,n)}-\frac{a_{n+2}}{f(n+2,n+1)}=\frac{a_{n+1}f(n+2,n+1)-a_{n+2}f(n+1,n)}{f(n+1,n)f(n+2,n+1)}.
$$
Next, note that $f(n+1,n)=-2acn+f_n$ with $f_n=-(ac+cb+ad)+e_{n+1,n}$ a bounded sequence, $|f_n|\leq f:=|ac+cb+ad|+e$
while $a_n=\ell+\tilde a_n$ with $|\tilde a_n|\leq Cn^{-1}$. But then
\begin{multline*}
a_{n+1}f(n+2,n+1)-a_{n+2}f(n+1,n)\\=(\ell+\tilde a_{n+1})(-2acn-2ac+f_{n+1})-(\ell+\tilde a_{n+2})(-2acn+f_{n})\\
=2acn(\tilde a_{n+2}-\tilde a_{n+1})+(\ell+\tilde a_{n+1})(-2ac+f_{n+1})-(\ell+\tilde a_{n+2})f_n
\end{multline*}
which is bounded by $F:=ac(6C+2|\ell|)+2(|\ell|+C)f$. As for $n\geq f/ac$, $f(n+1,n)\leq -acn$ we obtain
$$
\abs{\frac{a_{n+1}}{f(n+1,n)}-\frac{a_{n+2}}{f(n+2,n+1)}}\leq \frac{F}{(ac)^2}n^{-2}
$$
as claimed.
\end{proof}

\subsection{A simple criteria for condition $(D)$}

In the examples we have in mind, condition $(D)$ will be very easy to check. Indeed, it will fall in the scope of the following simple lemma:

\begin{lemma}
\label{lem:condD}
Assume that the following conditions hold:
\begin{enumerate}
\renewcommand{\theenumi}{\roman{enumi}}
\item $\Omega\subset\R^d$ is an open set and the set of smooth compactly supported functions $\cc^\infty_c(\Omega)$ is dense in every
$L^p(\Omega)$, $1<p<\infty$;

\item there exists a differential operators $L$ (resp. $\tilde L$) such that each $\ffi_n$ (resp. $\psi_n$'s) is an eigenfunctions of $L$ (resp. $\tilde L$);

\item writing $L\ffi_n=\lambda_n\ffi_n$ (resp. $\tilde L\psi_n=\tilde\lambda_n\psi_n$) we further assume that
there is an $\alpha>0$ (resp. $\tilde\alpha>0$) such that
$\lambda_n\gtrsim n^\alpha$ (resp. $\tilde\lambda_n\gtrsim n^\alpha$) when $n$ is big enough;

\item $(\ffi_n)$ (resp. $\psi_n$) satisfy condition $(L)$.
\end{enumerate}

Under the above conditions, $\Phi_N f\to f$ (resp. $\Psi_N f\to f$) in $L^p(\mu)$ for every $f\in\cc^\infty_c(\Omega)$.
\end{lemma}

\begin{proof} Indeed, if $f\in\cc^\infty_c(\Omega)$ and $n$ is big enough,
$$
\scal{\ffi_n,f}_{L^2(\mu)}=\frac{1}{\lambda_n}\scal{L\ffi_n,f}_{L^2(\mu)}
=\frac{1}{\lambda_n}\scal{\ffi_n,L^*f}_{L^2(\mu)}
=\frac{1}{\lambda_n^k}\scal{\ffi_n,(L^*)^kf}_{L^2(\mu)}
$$
by induction on $k$. But then
$|\scal{\ffi_n,f}_{L^2(\mu)}|\lesssim n^{-k\alpha}\norm{(L^*)^k f}_{L^2(\mu)}$.
As $\norm{\ffi_n}_{L^p(\mu)}\lesssim n^{\alpha_p}$ it is enough to take $k$ big enough to have $-k\alpha+\alpha_p<-1$
to see that
$$
\sum_{n\geq 0}\scal{\ffi_n,f}_{L^2(\mu)}\ffi_n
$$
converges in $L^p(\mu)$. As the limit of this series in $L^2(\mu)$ is $f$, so is the limit in $L^p(\mu)$.
The proof for $\psi_n$ is the same.
\end{proof}

\subsection{The Hilbert transform on weighted $L^p$ spaces}

In this section $1<p<\infty$.

First, let us recall that $\omega\,: J\to\R_+$ ($J$ an interval) is a Muckenhaupt $A^p$ weight if
$$
\ent{\omega}_{A^p}:=\left(\frac{1}{|K|}\int_K\omega(x)\,\mbox{d}x\right)\left(\frac{1}{|K|}\int_K\omega(x)^{-\frac{p'}{p}}\,\mbox{d}x\right)<+\infty
$$
where the supremum is taken over all intervals $K\subset J$.
The quantity $\ent{\omega}_{A^p}$ is called the $A^p$-characteristic of $\omega$ (or $A^p$ norm, though it is not a norm).

Let us recall that the Hilbert transform is defined as
$$
\hh f(x)=\frac{1}{\pi}\int_J\frac{f(y)}{x-y}\,\mathrm{d}y
$$
where the integral has to be taken in the principal value sense.

Hunt, Muckenhaupt and Wheeden \cite{HMW} proved that the Hilbert transform extends into a bounded linear operator
$L^p(J,\omega)\to L^p(J,\omega)$ if and only if $\omega$ is an $A^p$ weight and the sharp dependence
on the $A^p$ characteristic has been obtained by Petermichl:

\begin{theorem}[Petermichl] Let $1<p<+\infty$, $J$ an interval
and let $\omega$ be an $A_p$ weight, then
\begin{equation}
\label{eq:hilweight}
\norm{\hh}_{L^p(J,\omega)\to L^p(J,\omega)}\lesssim \ent{\omega}_{A^p}^{\max(1,(p-1)^{-1})}.
\end{equation}
\end{theorem}

Let us now estimate some $A^p$ characteristics that we will need in the sequel, when considering the Hankel prolates:

\begin{lemma}
\label{lem:ap}
Let $1<p<\infty$, $,\alpha\in\R$ and $\mu\geq1$. Let $\omega_{\alpha,\pm}$ be defined by
\begin{equation}
\label{eq:defomega34}
\omega_{\alpha,\pm}(x)=x^{\alpha}\big(|c\sqrt{x}-\mu|+\mu^{\frac{1}{3}}\big)^{\pm\frac{p}{4}}.
\end{equation}
Then $x^\alpha\in A^p[0,1]$ if and only if $x^\alpha\in A_p[1,+\infty]$ if and only if $-1<\alpha<p-1$. 
Moreover, $\omega_{\alpha,\pm}\in A^p[0,+\infty]$ if $-1+\frac{p}{8}<\alpha<\dst\frac{7}{8}p-1$ and in this case
$$
\ent{\omega_{\alpha,\pm}}_{A^p}\lesssim\begin{cases} 1&\mbox{if }\frac{4}{3}<p<4\\
\mu^{3/4}&\mbox{otherwise}\end{cases}
$$
with the implied constant depending on $\alpha$.
\end{lemma}

\begin{proof}[Proof of Lemma \ref{lem:ap}] The first part is well known and left to the reader.
Recall that if $\omega$ is an $A^p$ weight, then so is $\lambda \omega(\mu x)$ with
$\ent{\lambda\omega_j(\mu x)}_{A^p}=\ent{\omega}_{A^p}$.

Next, we have
$$
\omega_{\alpha,+}(x)=(c^2\mu^{-2}x)^{\alpha} c^{-2\alpha} \mu^{2\alpha-p/4}
\big(|\sqrt{c^2\mu^{-2}x}-1|+\mu^{-\frac{2}{3}}\big)^{\frac{p}{4}}=\lambda_3\tilde \omega_{\alpha,+}(c^2\mu^{-2}x)
$$
with $\lambda_3=c^{-2\alpha} \mu^{2\alpha-p/4}$ and
$$
\tilde\omega_{\alpha,+}(x)=x^{\alpha}\big(|\sqrt{x}-1|+\mu^{-\frac{2}{3}}\big)^{\frac{p}{4}}.
$$
So it is enough to estimate $\ent{\tilde\omega_{\alpha,+}}_{A^p}$. Similarly, we may replace
$\omega_{\alpha,-}$ by
$$
\tilde\omega_{\alpha,-}(x)=x^{\alpha}\big(|\sqrt{x}-1|+\mu^{-\frac{2}{3}}\big)^{-\frac{p}{4}}.
$$

Next, on $[0,1/2]$, $\tilde\omega_{\alpha,\pm}(x)\simeq x^{\alpha}\in A^p$ since $-1<\alpha<p-1$. 
Note that constants here are independent on $\mu\geq 1$. On the other hand, on
$[3/2,+\infty]$, $\tilde\omega_{\alpha,\pm}(x)\simeq x^{\alpha\pm p/8}\in A^p$ since $-1<\alpha\pm p/8<p-1$. 
Again, constants here are independent on $\mu\geq 1$.

Finally, on $[1/2,3/2]$,
$$
\tilde\omega_{\alpha,\pm}(x)\simeq \big(|\sqrt{x}-1|+\mu^{-\frac{2}{3}}\big)^{\pm\frac{p}{4}}\simeq
\omega_{\pm}(x):=\big(|x-1|+\mu^{-\frac{2}{3}}\big)^{\pm\frac{p}{4}}.
$$

We thus want to estimate
$$
\ent{\omega_{\pm}}_{A^p}=
\sup_I \left(\frac{1}{|I|}\int_I \big(|x-1|+\mu^{-\frac{2}{3}}\big)^{\pm\frac{p}{4}}\d x\right)
\left(\frac{1}{|I|}\int_I \big(|x-1|+\mu^{-\frac{2}{3}}\big)^{\mp\frac{p'}{4}}\d x\right)^{p/p'}
$$
where the sup runs over intervals $I\subset[1/2,3/2]$.
Equivalently, we want to estimate
$$
\ent{\omega_{\pm}}_{A^p}\simeq
\sup_I \left(\frac{1}{|I|}\int_I \big(|x|+\mu^{-\frac{2}{3}}\big)^{\pm\frac{p}{4}}\d x\right)
\left(\frac{1}{|I|}\int_I \big(|x|+\mu^{-\frac{2}{3}}\big)^{\mp\frac{p'}{4}}\d x\right)^{p/p'}
$$
where the sup runs over intervals $I\subset[-1/2,1/2]$. It is enough to consider $I=[0,a]$ then, when $p\not=4/3,4$, we are looking at
\begin{multline*}
\ent{\omega_{\pm,p}}_{A^p}\simeq
\sup_{a\in[0,1/2]}\left(\frac{\big(a+\mu^{-\frac{2}{3}}\big)^{1\pm\frac{p}{4}}-\big(\mu^{-\frac{2}{3}}\big)^{1\pm\frac{p}{4}}}{a}\right)
\left(\frac{\big(a+\mu^{-\frac{2}{3}}\big)^{1\mp\frac{p'}{4}}-\big(\mu^{-\frac{2}{3}}\big)^{1\mp\frac{p'}{4}}}{a}\right)^{p/p'}
\\
=\sup_{a\in[0,1/2]}\left(\frac{\big(1+a\mu^{\frac{2}{3}}\big)^{1\pm\frac{p}{4}}-1}{a\mu^{\frac{2}{3}}}\right)
\left(\frac{\big(1+a\mu^{\frac{2}{3}}\big)^{1\mp\frac{p'}{4}}-1}{a\mu^{\frac{2}{3}}}\right)^{p/p'}\\
=\sup_{t\in[0,\mu^{\frac{2}{3}}/2]}\left(\frac{\big(1+t\big)^{1\pm\frac{p}{4}}-1}{t}\right)
\left(\frac{\big(1+t\big)^{1\mp\frac{p'}{4}}-1}{t}\right)^{p/p'}
:=\sup_{t\in[0,\mu^{\frac{2}{3}}/2]} \ffi_\pm(t).
\end{multline*}

Note that $\ffi_\pm$ extends continuously at $0$ and that, when $t\to+\infty$. Moreover,

---  $\ffi_\pm(t)=O(1)$ for $p,q<4$ that is $4/3<p<4$

--- for $p>4$, $\ffi_-=\dst O\big(t^{\frac{p}{4}-1}\big)$, $\ffi_+=O(1)$,

--- for $p<4$, $\ffi_+=\dst O\big(t^{\frac{p'}{4}-1}\big)$, $\ffi_-=O(1)$.

The computation has to be slightly modified for $p=4/3$ to obtain $\ffi_+=O(\log t)$ and $\ffi_-=O(1)$ while
for $p=4$ one gets $\ffi_-=O(\log t)$, $\ffi_+=O(1)$. The result follows.
\end{proof}

\section{Application to weighted prolates}

\subsection{Weighted prolates}

In this section, we will fix real numbers $c>0$ and $ \alpha> 0$. We denote by $I=[-1,1]$
that will be endowed with the measure $\omega_\alpha(x)\,\mathrm{d}x$ with $\omega_{\alpha}(x)=(1-x^2)^{\alpha}$.
We will simply write $\omega_\alpha$ for the measure $\omega_\alpha(x)\,\mathrm{d}x$.
The aim of this section is to consider the set of Weighted Prolate Spheroidal Wave Functions (WPSWFs)
introduced in \cite{Karoui-Souabni1,Karoui-Souabni2, Wang2} and to study the $L^p(I,\omega_\alpha)$ convergence of the associated series.

More precisely, the WPSWFs are the eigenfunctions of the weighted finite Fourier transform operator
$\mathcal F_c^{(\alpha)}$ defined  by
\begin{equation}\label{Eq1.1}
\mathcal F_c^{(\alpha)} f(x)=\int_{-1}^1 e^{icxy}  f(y)\,\omega_{\alpha}(y)\,\mathrm{d}y.
\end{equation}
It is well known, see \cite{Karoui-Souabni1, Wang2} that the operator
$$
\mathcal Q_c^{(\alpha)}=\frac{c}{2\pi}
\mathcal F_c^{({\alpha})^*} \circ \mathcal F_c^{(\alpha)}
$$
is  defined on $L^2(I,\omega_{\alpha})$ by
\begin{equation}\label{EEq0}
\mathcal Q_c^{(\alpha)} g (x) = \int_{-1}^1 \frac{c}{2 \pi}\mathcal K_{\alpha}(c(x-y)) g(y) \omega_{\alpha}(y) \d y
\end{equation}
with
$$
\mathcal K_{\alpha}(x)=\sqrt{\pi} 2^{\alpha+1/2}\Gamma(\alpha+1) \frac{J_{\alpha+1/2}(x)}{x^{\alpha+1/2}}
$$
and $J_{\alpha}(\cdot)$ is the Bessel function of the first kind and order $ \alpha$.

It has been shown in \cite{Karoui-Souabni1, Wang2} that the last two integral operators commute with the following
Sturm-Liouville operator $\mathcal L_c^{(\alpha)}$ defined by
\begin{equation}\label{diff_oper1}
\mathcal L_c^{(\alpha)} (f)(x)= - \frac{d}{dx}\left[ \omega_{\alpha}(x) (1-x^2) f'(x)\right] +c^2 x^2 \omega_{\alpha}(x) f(x).
\end{equation}
Also, note that  the $(n+1)-$th eigenvalue $\chi_n(c)$ of $\mathcal L_c^{(\alpha)}$ satisfies the following classical inequalities,
\begin{equation}
\label{boundschi}
n (n+2\alpha+1) \leq \chi_n(c) \leq n (n+2\alpha+1) +c^2,\quad \forall n\geq 0.
\end{equation}
We will denote by $(\Psi_{n,c}^{(\alpha)})_{n\geq 0}$ the set of common eigenfunctions of $\mathcal F_c^{(\alpha)}, \mathcal Q_c^{(\alpha)}$ and $\mathcal L_c^{(\alpha)}$ and call them {\em Weighted Prolate Spheroidal Wave Functions (WPSWFs)}.
Then $\{ \ps , n\geq 0 \} $ is an orthogonal basis of $ L^2(I,\omega_{\alpha})$.

Our aim will be to apply Theorem \ref{th:main} with the following setting: $\Omega=I$, $\mu=\omega_\alpha$,
$\psi_n=\Psi_{n,c}^{(\alpha)}$. The first task will be to define the basis $\ffi_n$ and then to show that it satisfies
each of the desired properties.

\subsection{Some facts about Jacobi polynomials}

\subsubsection{Jacobi polynomials}

In this section, we gather results on Jacobi polynomials\footnote{We only use a particular subfamily of Jacobi polynomials and may as well call them ultra-spherical or Gegenbauer polynomials.} that will be used later.
The Jacobi polynomials are defined as being the orthonormal family of polynomials with respect to the scalar product
associated to $\norm{\cdot}_{L^2(I,\omega_\alpha)}$ with leading coefficient being non-negative.

Alternatively, we define the (non-normalized) Jacobi polynomials $P_k^{(\alpha)}$ through the induction
formula (see for example \cite{An})
\begin{equation}
\label{eq:recGeg}
P_{k+1}^{(\alpha)}(x)= A_k x P_{k}^{(\alpha)}(x) -C_k P_{k-1}^{(\alpha)}(x),\quad x\in [-1,1],
\end{equation}
where  $P_0^{(\alpha)}(x)=1,\quad P_1^{(\alpha)}(x)=(\alpha+1)x +\alpha$ and
$$
A_k =\frac{(2k+2\alpha+1)(k+\alpha+1)}{(k+1)(k+2\alpha+1)}=2-\frac{1}{k}+O(k^{-2})\ ,\
C_k=\frac{(k+\alpha)(k+\alpha+1)}{(k+1)(k+2\alpha+1)}=1-\frac{1}{k}+O(k^{-2}).
$$
We consider the normalized Jacobi polynomials $\wJ_k=\norm{P_k^{(\alpha)}}_{L^2(I,\omega_\alpha)}^{-1}P_k^{(\alpha)}$
which form an orthonormal basis of $L^2(I,\omega_{\alpha})$. A cumbersome computation shows that
\begin{equation*}
\wJ_{k}(x)= \frac{1}{\sqrt{h_k^{(\alpha)}}}\J_k(x),\quad h_k^{(\alpha)}=\frac{2^{2\alpha+1}\Gamma(k+\alpha+1)^2}{k!(2k+2\alpha+1)\Gamma(k+2\alpha+1)}.
\end{equation*}
The normalized Jacobi polynomials satisfy the recursion formula
\begin{equation}\label{recursion}
\wJ_{k+1}(x)= \tilde A_k x \wJ_{k}(x) -\tilde C_k \wJ_{k-1}(x),
\end{equation}
where
\begin{equation}\label{coefficients}
\tilde A_k=\sqrt{\frac{h_{k}^{(\alpha)}}{h_{k+1}^{(\alpha)}}} A_k=2+O(k^{-2})
\quad,\quad
\tilde C_k=\sqrt{\frac{h_{k-1}^{(\alpha)}}{h_{k+1}^{(\alpha)}}} C_k=1-\frac{1}{2k}+O(k^{-2})
\end{equation}
since
$$
\frac{h_{k}^{(\alpha)}}{h_{k+1}^{(\alpha)}}=\frac{(k+1)(2k+2\alpha+3)(k+2\alpha+1)}{(2k+2\alpha+1)(k+\alpha+1)^2}=1+\frac{1}{k}+O(k^{-2}).
$$

Further, it has been shown that
\begin{equation}
\label{eq:Jacobibound}
|\wJ_{n}(x)|\lesssim w_{n,\alpha}(x):=(\sqrt{1-x}+n^{-1})^{-\alpha-1/2}(\sqrt{1+x}+n^{-1})^{-\alpha-1/2}
\end{equation}
uniformly over $(-1,1)$ where the constant involved is independent of $n$ ({\it see e.g.} \cite[Chapter 4]{Szego}).

Moreover, let $p_0=2-\frac{1}{\alpha+3/2}$ so that $p_0^{\prime}=\dst 2+\frac{1}{\alpha+1/2}$ then,
for $1<p<\infty$ , the $L^p$-norm of Jacobi polynomials is given by Aptekarev, Buyarov and Degeza \cite{ABD} ({\it see also} \cite{ADMF}):
\begin{equation}
\label{normelp}
\|{\wJ_n}\|_{L^p(I,\omega_{\alpha})}= \begin{cases}
C(\alpha, p) + \circ(1) &\mbox{if }1<p<p_0^{\prime}\\
C(\alpha, p) \log(n)(1+\circ(1))&\mbox{when }p=p_0^{\prime}\\
n^{(\alpha+1/2)(p-p_0^{\prime})}&\mbox{when }p>p_0^{\prime}
\end{cases}
\end{equation}
with $C(\alpha,p)$ is a generic constant depending only on $\alpha$ and $p$.
Note that
\begin{equation}
\label{normelplq}
L_n(\alpha) := \|\wJ_n\|_{L^p(I,\omega_{\alpha})}\|{\wJ_n}\|_{L^{p'}(I,\omega_{\alpha})} \approx
\begin{cases}
n^{(\alpha+1/2)(p^{\prime}-p_0^{\prime})}&\mbox{when }1<p<p_0\\
\log n&\mbox{when }p=p_0\mbox{ or }p=p_0^{\prime}\\
1&\mbox{when }p_0<p<p_0^{\prime}\\
n^{(\alpha+1/2)(p-p_0^{\prime})}&\mbox{when }p>p_0^{\prime}.
\end{cases}
\end{equation}
In particular, $L_n(\alpha)= O(n^{\alpha_p})$ with $\alpha_p=0$ if $p\in(p_0,p_0^{\prime})$ and $\alpha_p<1$
when $p\in(p_1,p_1^{\prime})$ with $p_1^{\prime}=\dst p_0^{\prime}+\frac{1}{\alpha+1/2}
=2+\frac{2}{\alpha+1/2}$.
It follows that Condition $(L)$ of Theorem \ref{th:main} is satisfied.

Further, the Jacobi polynomials are eigenfunctions of the differential operator
$$
Lf:=(1-x^{2})f''-(2\alpha+1)xf'
$$
with eigenvalue $\lambda_n=-n(n+2\alpha+1)$. It follows from Lemma \ref{lem:condD} that Condition $(D)$ of Theorem \ref{th:main}
is also satisfied.

\subsubsection{The Projection on the span of Jacobi polynomials}
Let us now introduce
$$
C^{(\alpha)}_N(x,y)= \sum_{k=0}^N \wJ_k(x) \wJ_k(y)
$$
and, according to the Christofel Darboux Formula,
$$
C^{(\alpha)}_N(x,y)= \frac{\beta_N}{\beta_{N+1}}\frac{\wJ_{N+1}(x)\wJ_N(y)-\wJ_{N+1}(y)\wJ_N(x)}{x-y}.
$$
Pollard \cite{Pollard2} proved that $C^{(\alpha)}_N$ defines a bounded operator $C^{(\alpha)}_N\,:L^p(I,\omega_\alpha)
\to L^p(I,\omega_\alpha)$ and that the
operators $C^{(\alpha)}_N$ are uniformly bounded in the range $p_0<p<p_0^\prime$.
Further, he proved that the series
$C^{(\alpha)}_Nf$ may diverge if $p\notin [p_0,p_0^\prime]$ but did not provide a bound for $C^{(\alpha)}_N$. The divergence at the end points was proved later by Newman and Rudin \cite{Newman}.
The key point in Pollard's proof is the following identity
\begin{eqnarray*}
C^{(\alpha)}_Nf(x)&=&
U_n\widetilde{P}_{n+1}^{(\alpha)}(x)\int_{-1}^1\frac{\widetilde{Q}_{n}^{(\alpha)}(y)f(y)\omega_\alpha(y)}{x-y}\,\mbox{d}y\\
&&+V_n\widetilde{Q}_{n}^{(\alpha)}(x)\int_{-1}^1\frac{\widetilde{P}_{n+1}^{(\alpha)}(y)f(y)\omega_\alpha(y)}{x-y}\,\mbox{d}y\\
&&+W_n\scal{f,\widetilde{P}_{n+1}^{(\alpha)}}_{L^2(I,\omega_\alpha)}\widetilde{P}_{n+1}^{(\alpha)}(x)\\
&=&C^{(\alpha,1)}_Nf(x)+C^{(\alpha,2)}_Nf(x)+C^{(\alpha,3)}_Nf(x)
\end{eqnarray*}
where $U_n,V_n,W_n\to\frac{1}{2}$ and $\widetilde{Q}_{n}^{(\alpha)}$ is an other family of orthogonal polynomials.

H\"older's inequality and Lemma \ref{lem:triv} show that
$\norm{C^{(\alpha,3)}_N}_{L^p(I,\omega_\alpha)\to L^p(I,\omega_\alpha)}\lesssim N^{\alpha_p}$ while
Pollard showed that $\norm{C^{(\alpha,j)}_Nf}_{L^p(I,\omega_\alpha)\to L^p(I,\omega_\alpha)}\lesssim 1$ for $j=1,2$.

Let us summarize the results from this section

\begin{lemma}
\label{lem:projjacobi}
Let $1<p<\infty$ and $\alpha>-1/2$, $\eps>0$. Let $I=(-1,1)$, $\omega_\alpha(x)=(1-x^2)^\alpha$ and
$\tilde P_n^{(\alpha)}$ be the Jacobi polynomials , {\it i.e.} the orthonormal family of polynomials
in $L^2(I,\omega_\alpha)$ defined above. Let $C^{(\alpha)}_N$ be the orthogonal projection on the span
of $\tilde P_0^{(\alpha)},\ldots,\tilde P_N^{(\alpha)}$.

Let $p_0=2-\frac{1}{\alpha+3/2}$ so that $p_0^{\prime}=\dst 2+\frac{1}{\alpha+1/2}$. Define
$$
\alpha_p=\begin{cases}
(\alpha+1/2)(p^{\prime}-p_0^{\prime})&\mbox{ when }1<p<p_0\\
\eps&\mbox{ when }p=p_0\mbox{ or }p_0^{\prime}\\
0&\mbox{ when } p\in(p_0,p_0^{\prime})\\
(\alpha+1/2)(p-p_0^{\prime})&\mbox{ when }p>p_0^{\prime}
\end{cases}
$$
so that $\alpha_p<1$  when $p\in(p_1,p_1^{\prime})$ with $p_1^{\prime}=\dst p_0^{\prime}+\frac{1}{\alpha+1/2}
=2+\frac{2}{\alpha+1/2}$. Then

--- {\rm Aptekarev, Buyarov and Degeza \cite{ABD}} we have
\begin{equation}
\label{normlplq}
\|\wJ_n\|_{L^p(I,\omega_{\alpha})}\|{\wJ_n}\|_{L^{p'}(I,\omega_{\alpha})}\lesssim n^{\alpha_p};
\end{equation}
--- {\rm Pollard \cite{Pollard2}} the operators $C^{(\alpha)}_N$ extend to bounded operators
$L^p(I,\omega_\alpha)\to L^p(I,\omega_\alpha)$ with
$$
\norm{C^{(\alpha)}_N}_{L^p(I,\omega_\alpha)\to L^p(I,\omega_\alpha)}\lesssim N^{\alpha_p}.
$$
\end{lemma}

\subsection{Condition $(R)$}
The aim of this section is to establish condition $(R)$ of Theorem \ref{th:main}.

The series expansion of the
WPSWFs in the basis of Jacobi polynomials $(\widetilde P_n^{(\alpha)})$ which can be written in the form
\begin{equation}\label{expansion1}
\Psi^{(\alpha)}_{n,c}=\sum_{k\geq 0} \beta_k^n \widetilde P_k^{(\alpha)}
\end{equation}
where $\beta_k^n=\scal{\Psi^{(\alpha)}_{n,c},\widetilde P_k^{(\alpha)}}_{L^2(I, \omega_{\alpha})}$.
By replacing the expression (\ref{expansion1}) in the differential equation \eqref{diff_oper1}, one gets the following recursion formula satisfied by the $\beta_k^n$ for $k\geq 2$
\begin{equation}
\label{rec:WPSFs}
f(k,n) \beta_k^{(n)} = a_k^{(\alpha)}\beta_{k-2}^{(n)} + a_{k+2}^{(\alpha)}\beta^{(n)}_{k+2},
\end{equation}
where
\begin{eqnarray}\label{coeff_GPSWFs}
 f(k,n)&=&\frac{ \chi_n(c) - \Bigg( k(k+2\alpha+1)+c^2 b_k^{(\alpha)}   \Bigg)}{c^2}   \\
 a_k^{(\alpha)}&=& \frac{\sqrt{k(k-1)(k+2\alpha)(k+2\alpha-1)}}{(2k+2\alpha-1)\sqrt{(2k+2\alpha+1)(2k+2\alpha-3)}}   \nonumber \\
 b_k^{(\alpha)}&=& \frac{2k(k+2\alpha+1)+2\alpha-1}{(2k+2\alpha+3)(2k+2\alpha-1)} . \nonumber
\end{eqnarray}

This is not exactly of the desired form. To overcome this problem, first note that $\Psi_{n,c}^{(\alpha)}$ and $\widetilde P_{n}^{(\alpha)}$ have same parity as $n$, so that $\beta_k^{(n)}=0$ if $k$ and $n$ have opposite parity.
Next, we decompose
$$
L^p(I, \omega_{\alpha})=L^p_e(I, \omega_{\alpha})\oplus L^p_o(I, \omega_{\alpha})
$$
where $L^p_e(I, \omega_{\alpha})$, resp. $L^p_o(I, \omega_{\alpha})$, is the set of even, resp. odd, functions
in $L^p(I, \omega_{\alpha})$.

Our aim is then to characterize for which $p$, for every $f\in L^p_e(I, \omega_{\alpha})$ --- resp. $f\in L^p_o(I, \omega_{\alpha})$ ---
$\sum_{n\geq 0} \scal{f,\Psi_{2n,c}^{(\alpha)}}\Psi_{n,c}^{(\alpha)}$
--- resp. $\sum_{n\geq 0} \scal{f,\Psi_{2n+1,c}^{(\alpha)}}\Psi_{n,c}^{(\alpha)}$ ---
converges to $f$ in $L^p(I, \omega_{\alpha})$.
This can be done by applying Theorem \ref{th:main}.
To do so, we will now establish condition $(R)$.

First note that $a_k^{(\alpha)}\to 1/4$ and that we may write
$$
a_k^{(\alpha)}=\frac{\sqrt{(1-k^{-1})(1+2\alpha k^{-1})\bigl(1+(2\alpha-1)k^{-1}\bigr)}}{\bigl(2+(2\alpha-1)k^{-1}\bigr)\sqrt{\bigl(2+(2\alpha+1)k^{-1}\bigr)\bigl(2+(2\alpha-3)k^{-1}\bigr)}}
$$
from which it is obvious that $a_k^{(\alpha)}=1/4+O(k^{-1})$. As $b_k^{(\alpha)}$ is clearly bounded,
all conditions of Lemma \ref{lem:fkn} are satisfied and $f(k,n)$ satisfies all requirements of condition $(R)$.

It remains to establish the following:

\begin{lemma}
For every $\alpha>-1/2$ and every $k\geq 2$, $|a_k^{(\alpha)}|\leq 1/2$.
\end{lemma}

\begin{proof} First
$$
a_2^{(\alpha)}=\frac{2\sqrt{1+\alpha}}{(3+2\alpha)\sqrt{5+2\alpha}}
$$
which is maximal for $\alpha=\sqrt{2}-2<-1/2$ and the maximal value is $\dst\sqrt{\frac{64\sqrt{2}-52}{343}}\sim 0.335<1/2$.

Next, for $k\geq 3$ and $-1/2<\alpha\leq 1/2$, we bound
$$
a_k^{(\alpha)}\leq \frac{1}{4}\sqrt{\frac{k(k+1)}{(k-1)(k-2)}}\leq\frac{\sqrt{3}}{4}<\frac{1}{2}.
$$
Finally, $2k+2\alpha+1\geq k+2\alpha$, for $\alpha>0$, $(2k+2\alpha-1)\geq 2\sqrt{k(k-1)}$,
and $2k+2\alpha-3\geq k+2\alpha-1$ when $k\geq 2$ thus
$$
a_k^{(\alpha)}= \frac{\sqrt{k(k-1)(k+2\alpha)(k+2\alpha-1)}}{(2k+2\alpha-1)\sqrt{(2k+2\alpha+1)(2k+2\alpha-3)}}
<\frac{1}{2}
$$
as announced.
\end{proof}

\subsection{Condition $(B')$}

We will now establish Condition $(B')$ in Theorem \ref{th:main}. As we consider separately $L^p_e(I, \omega_{\alpha})$
and $L^p_o(I, \omega_{\alpha})$, we actually have to estimate the $L^p(I, \omega_{\alpha})\to L^p(I, \omega_{\alpha})$
norm of the operator with kernel
$$
\Phi_N(x,y)=\sum_{n=0}^N
\bigl(\widetilde{P}_n^{(\alpha)}(x)\widetilde{P}_{n+2}^{(\alpha)}(y)+\widetilde{P}_{n+2}^{(\alpha)}(x)\widetilde{P}_n^{(\alpha)}(y)\bigr)
$$
with $\widetilde{P}_{n+2}^{(\alpha)}$ instead of $\widetilde{P}_{n+1}^{(\alpha)}$.
We will also write $\Phi_N$ for the associated operator on $L^p(I, \omega_{\alpha})$. Note that the bound \eqref{normlplq}
together with Lemma \ref{lem:triv} leads to
$$
\norm{\Phi_N}_{L^p(I, \omega_{\alpha})\to L^p(I, \omega_{\alpha})}\lesssim N^{1+\alpha_p}
$$
which is not good enough for our needs.

Using the recursion formula \eqref{recursion} twice, we get for $n\geq 2$,
\begin{eqnarray*}
\widetilde{P}_n^{(\alpha)}(x)\widetilde{P}_{n+2}^{(\alpha)}(y)&=&
\widetilde{P}_n^{(\alpha)}(x)\bigl(\tilde A_{n+1} y\widetilde{P}_{n+1}^{(\alpha)}(y)
-\tilde C_{n+1}\widetilde{P}_{n}^{(\alpha)}(y)\bigr)\\
&=&y\tilde A_{n+1}\widetilde{P}_n^{(\alpha)}(x)\bigl(y\tilde A_{n}\widetilde{P}_n^{(\alpha)}(y)
-\tilde C_n\widetilde{P}_{n-1}^{(\alpha)}(y)\bigr)
-\tilde C_{n+1}\widetilde{P}_n^{(\alpha)}(x)\widetilde{P}_{n}^{(\alpha)}(y)\\
&=&y^2\tilde A_{n+1}\tilde A_n\widetilde{P}_n^{(\alpha)}(x)\widetilde{P}_n^{(\alpha)}(y)
-y\tilde A_{n+1}\tilde C_n\widetilde{P}_n^{(\alpha)}(x)\widetilde{P}_{n-1}^{(\alpha)}(y)
-\tilde C_{n+1}\widetilde{P}_n^{(\alpha)}(x)\widetilde{P}_{n}^{(\alpha)}(y).
\end{eqnarray*}
Next note that
$$
y\widetilde{P}_{n-1}^{(\alpha)}(y)=\frac{1}{\tilde A_{n-1}}\widetilde{P}_n^{(\alpha)}(y)
+\frac{\tilde C_{n-1}}{\tilde A_{n-1}}\widetilde{P}_{n-2}^{(\alpha)}(y)
$$
so that
\begin{multline*}
\widetilde{P}_n^{(\alpha)}(x)\widetilde{P}_{n+2}^{(\alpha)}(y)
=y^2\tilde A_{n+1}\tilde A_n\widetilde{P}_n^{(\alpha)}(x)\widetilde{P}_n^{(\alpha)}(y)
-\left(\tilde C_{n+1}+\frac{\tilde A_{n+1}\tilde C_n}{\tilde A_{n-1}}\right)\widetilde{P}_n^{(\alpha)}(x)\widetilde{P}_n^{(\alpha)}(y)\\
-\frac{\tilde A_{n+1}\tilde C_n\tilde C_{n-1}}{\tilde A_{n-1}}\widetilde{P}_n^{(\alpha)}(x)\widetilde{P}_{n-2}^{(\alpha)}(y).
\end{multline*}

Let us define 
\begin{eqnarray*}
\kappa_n&=&\dst-\Big(\tilde C_{n+1}+\frac{\tilde A_{n+1}\tilde C_n}{\tilde A_{n-1}}\Big)=-1+\frac{1}{n}+O(n^{-2})\\
\tilde\kappa_n&=&\dst 1-\frac{\tilde A_{n+1}\tilde C_n\tilde C_{n-1}}{\tilde A_{n-1}}=\frac{1}{n}+O(n^{-2}).
\end {eqnarray*}
Then
\begin{eqnarray*}
\widetilde{P}_n^{(\alpha)}(x)\widetilde{P}_{n+2}^{(\alpha)}(y)+\widetilde{P}_n^{(\alpha)}(y)\widetilde{P}_{n+2}^{(\alpha)}(x)
&=&y^2\tilde A_{n+1}\tilde A_n\widetilde{P}_n^{(\alpha)}(x)\widetilde{P}_n^{(\alpha)}(y)
+\kappa_n\widetilde{P}_n^{(\alpha)}(x)\widetilde{P}_n^{(\alpha)}(y)\\
&&+\tilde\kappa_n\widetilde{P}_n^{(\alpha)}(x)\widetilde{P}_{n-2}^{(\alpha)}(y).
\end{eqnarray*}
Summing over $n$, we conclude that
\begin{eqnarray*}
\Phi_N(x,y)&=&\widetilde{P}_0^{(\alpha)}(x)\widetilde{P}_{2}^{(\alpha)}(y)
+\sum_{n=2}^N\bigl(\widetilde{P}_n^{(\alpha)}(x)\widetilde{P}_{n+2}^{(\alpha)}(y)+\widetilde{P}_n^{(\alpha)}(x)\widetilde{P}_{n-2}^{(\alpha)}(y)\bigr)-\widetilde{P}_N^{(\alpha)}(x)\widetilde{P}_{N+2}^{(\alpha)}(y)\\
&=&\widetilde{P}_0^{(\alpha)}(x)\widetilde{P}_{2}^{(\alpha)}(y)-\widetilde{P}_N^{(\alpha)}(x)\widetilde{P}_{N+2}^{(\alpha)}(y)
+y^2\sum_{n=2}^N\tilde A_{n+1}\tilde A_n\widetilde{P}_n^{(\alpha)}(x)\widetilde{P}_n^{(\alpha)}(y)\\
&&+\sum_{n=2}^N\kappa_n\widetilde{P}_n^{(\alpha)}(x)\widetilde{P}_n^{(\alpha)}(y)
+\sum_{n=2}^N\tilde\kappa_n\widetilde{P}_n^{(\alpha)}(x)\widetilde{P}_{n-2}^{(\alpha)}(y).
\end{eqnarray*}
Further, exchanging the roles of $x$ and $y$ and summing, we obtain
$2\Phi_N(x,y)=\Phi_N^1(x,y)+\cdots+\Phi_N^6(x,y)$ where
\begin{eqnarray*}
\Phi_N^1(x,y)&=&\widetilde{P}_0^{(\alpha)}(x)\widetilde{P}_{2}^{(\alpha)}(y)+\widetilde{P}_2^{(\alpha)}(x)\widetilde{P}_0^{(\alpha)}(y)\\
&&-(4x^2+4y^2-2)\bigl(\widetilde{P}_0^{(\alpha)}(x)\widetilde{P}_0^{(\alpha)}(y)+\widetilde{P}_1^{(\alpha)}(x)\widetilde{P}_1^{(\alpha)}(y)\bigr)\\
\Phi_N^2(x,y)&=&-\widetilde{P}_N^{(\alpha)}(x)\widetilde{P}_{N+2}^{(\alpha)}(y)-\widetilde{P}_{N+2}^{(\alpha)}(x)\widetilde{P}_N^{(\alpha)}(y)\\
\Phi_N^3(x,y)&=&(x^2+y^2)\sum_{n=2}^N\tilde A_{n+1}\tilde A_n\widetilde{P}_n^{(\alpha)}(x)\widetilde{P}_n^{(\alpha)}(y)\\
\Phi_N^4(x,y)&=&2\sum_{n=2}^N\kappa_n\widetilde{P}_n^{(\alpha)}(x)\widetilde{P}_n^{(\alpha)}(y)\\
\Phi_N^5(x,y)&=&\sum_{n=2}^N\tilde\kappa_n\bigl(\widetilde{P}_n^{(\alpha)}(x)\widetilde{P}_{n-2}^{(\alpha)}(y)
+\widetilde{P}_{n-2}^{(\alpha)}(x)\widetilde{P}_n^{(\alpha)}(y)\bigr).
\end{eqnarray*}

We also write $\Phi_N^j$ for the corresponding integral operators and will now estimate their norm as operators $L^p(I,\omega_\alpha)
\to L^p(I,\omega_\alpha)$.

Using the bound \eqref{normlplq} together with Lemma \ref{lem:triv} we get
$$
\norm{\Phi_N^1}_{L^p(I, \omega_{\alpha})\to L^p(I, \omega_{\alpha})}\lesssim 1
$$
and
$$
\norm{\Phi_N^2}_{L^p(I, \omega_{\alpha})\to L^p(I, \omega_{\alpha})}\lesssim N^{\alpha_p}.
$$

Using Abel summation, we can write
\begin{eqnarray*}
\Phi_N^3(x,y)&=&-\tilde A_3\tilde A_2 (x^2+y^2)C_1^{(\alpha)}(x,y)+
(x^2+y^2)\sum_{n=2}^N\tilde A_{n+1}(\tilde A_n-\tilde A_{n+2})C_n^{(\alpha)}(x,y)\\
&&+\tilde A_{N+1}\tilde A_N (x^2+y^2)C_N^{(\alpha)}(x,y)\\
&=&\Phi_N^{3,1}(x,y)+\Phi_N^{3,2}(x,y)+\Phi_N^{3,3}(x,y).
\end{eqnarray*}
Of course
$$
\norm{\Phi_N^{3,1}}_{L^p(I, \omega_{\alpha})\to L^p(I, \omega_{\alpha})}\lesssim 1
$$
while Lemma \ref{lem:projjacobi} shows that
$$
\norm{\Phi_N^{3,2}}_{L^p(I, \omega_{\alpha})\to L^p(I, \omega_{\alpha})}\lesssim
\sum_{n=2}^N \frac{1}{n^2}n^{\alpha_p}\lesssim N^{\alpha_p-1}
$$
since $|\tilde A_{n+1}(\tilde A_n-\tilde A_{n+2})|\lesssim n^{-2}$
and
$$
\norm{\Phi_N^{3,3}}_{L^p(I, \omega_{\alpha})\to L^p(I, \omega_{\alpha})}\lesssim N^{\alpha_p}.
$$

Using Abel summation again, we write
\begin{eqnarray*}
\Phi_N^4(x,y)&=&-2\kappa_2 C_1^{(\alpha)}(x,y)+2\sum_{n=2}^N(\kappa_n-\kappa_{n+1})C_n^{(\alpha)}(x,y)
+2\kappa_N y^2C_N^{(\alpha)}(x,y)\\
&=&\Phi_N^{4,1}(x,y)+\Phi_N^{4,2}(x,y)+\Phi_N^{4,3}(x,y).
\end{eqnarray*}
Again
$$
\norm{\Phi_N^{4,1}}_{L^p(I, \omega_{\alpha})\to L^p(I, \omega_{\alpha})}\lesssim 1
$$
while Lemma \ref{lem:projjacobi} shows that
$$
\norm{\Phi_N^{4,2}}_{L^p(I, \omega_{\alpha})\to L^p(I, \omega_{\alpha})}\lesssim
\sum_{n=2}^N \frac{1}{n^2}n^{\alpha_p}\lesssim N^{\alpha_p-1}
$$
since $|\kappa_n-\kappa_{n+1}|\lesssim n^{-2}$
and
$$
\norm{\Phi_N^{4,3}}_{L^p(I, \omega_{\alpha})\to L^p(I, \omega_{\alpha})}\lesssim N^{\alpha_p}.
$$

A last use of Abel summation leads to
\begin{eqnarray*}
\Phi_N^5(x,y)&=&-\tilde\kappa_2 \Phi_1^{(\alpha)}(x,y)+\sum_{n=2}^N(\tilde\kappa_n-\tilde\kappa_{n+1})\Phi_n^{(\alpha)}(x,y)
+\tilde\kappa_N \Phi_N^{(\alpha)}(x,y)\\
&=&\Phi_N^{5,1}(x,y)+\Phi_N^{5,2}(x,y)+\Phi_N^{5,3}(x,y).
\end{eqnarray*}
Of course
$$
\norm{\Phi_N^{5,1}}_{L^p(I, \omega_{\alpha})\to L^p(I, \omega_{\alpha})}\lesssim 1.
$$
For the two other terms, we will use the fact that $\norm{\Phi_N}_{L^p(I, \omega_{\alpha})\to L^p(I, \omega_{\alpha})}\lesssim N^{1+\alpha_p}$ and that $\tilde\kappa_n=n^{-1}+O(n^{-2})$, in particular $|\tilde\kappa_n-\tilde\kappa_{n+1}|\lesssim n^{-2}$.
It follows that
$$
\norm{\Phi_N^{5,2}}_{L^p(I, \omega_{\alpha})\to L^p(I, \omega_{\alpha})}\lesssim \sum_{n=2}^Nn^{-2}n^{1+\alpha_p}\lesssim N^{\alpha_p}
$$
and
$$
\norm{\Phi_N^{5,3}}_{L^p(I, \omega_{\alpha})\to L^p(I, \omega_{\alpha})}\lesssim N^{-1}N^{1+\alpha_p}\lesssim N^{\alpha_p}.
$$

Summing all terms, Condition $(B)$ of Theorem \ref{th:main} is satisfied.

\subsection{Conclusion}

It remains to conclude, all conditions of Theorem \ref{th:main} are satisfied. Therefore, the Weighted prolate spheroidal series
converges in $L^p(I,\omega_\alpha)$ if and only if the Jacobi series converge. The later ones converge in $L^p(I,\omega_\alpha)$
if and only if $p\in(p_0,p_0^\prime)$. We have thus proved the following:

\begin{theorem}
Let $\alpha>-1/2$ and $c>0$, $N\geq 0$. Let $p_0=2-\frac{1}{\alpha+3/2}$ so that $p_0^{\prime}=\dst 2+\frac{1}{\alpha+1/2}$.

Let $(\psi_{n,c}^{(\alpha)})_{n\geq 0}$ be the family of weighted prolate spheroidal wave functions.
For a smooth function $f$ on $I=(-1,1)$, define
$$
\Psi^{(\alpha)}_Nf=\sum_{n=0}^N\scal{f,\psi_{n,c}^{(\alpha)}}_{L^2(I,\omega_\alpha)}\psi_{n,c}^{(\alpha)}.
$$
Then, for every $p\in(1,\infty)$, $\Psi^{(\alpha)}_N$ extends to a bounded operator $L^p(I,\omega_\alpha)\to L^p(I,\omega_\alpha)$.
Further
$$
\Psi^{(\alpha)}_Nf\to f\qquad \mbox{in }L^p(I,\omega_\alpha)
$$
for every $f\in L^p(I,\omega_\alpha)$ if and only if $p\in(p_0,p_0^\prime)$.
\end{theorem}

\section{Application to circular prolate spheroidal wave functions}

For two real numbers  $c>0$ and $\alpha>-\frac{1}{2}$, the family of the circular prolate spheroidal wave functions (CPSWFs), 
introduced by D. Slepian \cite{Slepian3}
and denoted by $\vp$, are the eigenfunctions of the finite Hankel transform  $\mathcal H_c^{\alpha}$, the operator on 
$L^2[0,1]$ with kernel given by $\mathcal H_c^{\alpha}(x,y)= \sqrt{cxy}J_{\alpha}(cxy)$. On other words
$$
\mathcal H_c^{\alpha}f(x)=\int_0^1
\sqrt{cxy}J_{\alpha}(cxy) f(y)\d y.
$$
We denote by $\mu_{n,\alpha}(c)$ the family of the eigenvalues of the operator $\mathcal H_c^{\alpha}$, that is $\mathcal H_c^{\alpha}\vp =
\mu_{n,\alpha}(c)\vp$.
The functions $\vp$ satisfy the following orthogonality relations:
$$
\int_{0}^{1}\psi_{n,c}^{\alpha}(x)\psi_{m,c}^{\alpha}(x)\d x
=\delta_{n,m}\quad\mbox{and}\quad
\int_{0}^{+\infty}\psi_{n,c}^{\alpha}(x)\psi_{m,c}^{\alpha}(x)\d x=
\frac{\delta_{n,m}}{c\mu_{n,\alpha}^2(c)}
$$
and the $\vp$'s constitute a complete orthonormal system in $L^2[0,1]$.

The $\vp$'s are also related to the Hankel operator $\mathcal{H}^{\alpha}$, the
integral operator on $L^2[0,+\infty[$ with kernel given by $\mathcal H^{\alpha}(x,y)=\sqrt{xy}J_{\alpha}(xy)$. More precisely,
$$
\mathcal H^{\alpha}(\psi_{n,c}^{\alpha})(x)=\frac{1}{c\mu_{n,\alpha}(c)}\psi_{n,c}^{\alpha}\left(\frac{x}{c}\right)\chi_{[0,c]}(x).
$$
According to Plancherel's theorem, the family
$\psi_{n,c}^{(\alpha)}=\sqrt{c}|\mu_{n,\alpha}(c)|\vp$ constitute a
complete orthonormal system in $\mathcal B_c^{\alpha}$ defined by:
\begin{equation}
\mathcal B_c^{\alpha}=\{ f\in L^2(0,\infty);\,\ \supp (\mathcal H^{\alpha}(f))\subset [0,c]\}.
\end{equation}
Fore more details, see for example \cite{Slepian3,Karoui-Boulsen}.
Our first aim in this section is to prove that in the case of CPSWFs, we have mean convergence in the Hankel Paley-Wiener space $B_{c,p}^{\alpha}$ defined by $$B_{c,p}^{\alpha}=\left\{f\in L^p(0,\infty); \supp(\mathcal{H^{\alpha}}(f))\subseteq[0,c]\right\},$$ if and only if $ 4/3< p <4.$

\subsection{Some facts about Spherical Bessel function}

\subsubsection{Spherical Bessel function}
The spherical Bessel function is defined as
\begin{equation}
\label{eq:sphbess}
\jc{n}(x)=\sqrt{2(2n+\alpha+1)}\frac{J_{2n+\alpha+1}(cx)}{\sqrt{cx}}.
\end{equation}
Here, $J_{\alpha}$ is the Bessel function of the first kind and order $\alpha.$
The spherical Bessel functions satisfy the orthogonality relation,
$$
\int_{0}^{+\infty}\jc{n}(x)\jc{m}(x)\d x=\delta_{n,m}.
$$
Their Hankel transforms are given by, see for example \cite{Slepian3}
\begin{equation}
\label{eq:besseljacobi}
\mathcal H^{\alpha}(\jc{n})(x)=\frac{\sqrt{2(2n+\alpha+1)}}{c}\left(\frac{x}{c}\right)^{\alpha+\frac{1}{2}}
P_n^{(\alpha,0)}\left(1-2\left(\frac{x}{c}\right)^2\right)
\chi_{[0,c]}(x).
\end{equation}
where $P_n^{(\alpha,0)}$ is the Jacobi polynomials of degree $n$ and
parameter $\alpha$, normalized so that
$P_n^{(\alpha,0)}(1)=\dst\frac{\Gamma(n+\alpha+1)}{\Gamma(n+1)\Gamma(\alpha+1)}$.
Introducing
\begin{equation}
\label{eq:tk}
T_{n,\alpha}(x)=(-1)^n\sqrt{2(2n+\alpha+1)}x^{\alpha+\frac{1}{2}}P_n^{(\alpha,0)}(1-2x^2).
\end{equation}
We thus get $\hh^{\alpha}(\jc{n})(x)=c^{-1}\chi_{[0,1]}(x/c) T_{n,\alpha}(x/c)$. Note that the orthogonality
relations of the $\jc{n}$'s and the unitary character of $\hh^{\alpha}$ imply that
$(T_{n,\alpha})_{n\ge 0}$ is an orthonormal basis of $L^2[0,1]$ while
the spherical Bessel functions $\jc{n}$ form a complete orthonormal system in $\mathcal B_c^{\alpha}$.

Further, using the induction property $\dst \frac{2\beta}{x}J_\beta(x)=J_{\beta-1}(x)+J_{\beta+1}(x)$,
we get the following induction formula
$$
\jc{n+1}=\frac{2\sqrt{(2n+\alpha+2)(2n+\alpha+3)}}{cx}\jc{n+1/2}-\frac{\sqrt{2n+\alpha+3}}{\sqrt{2n+\alpha+1}}\jc{n}.
$$
Moreover, for $1<p<\infty$, we have
\begin{equation}
	\label{eq:lpnorm}
	\norm{\jc{n}}_{L^p(0,\infty)} \sim \begin{cases}
		n^{-\frac{1}{2}+\frac{1}{p}}&\mbox{when }1<p<4\\
		n^{-\frac{1}{4}}\log n&\mbox{when }p=4\\
		n^{-\frac{1}{3}+\frac{1}{3p}}&\mbox{when }p>4
	\end{cases}.
\end{equation}
Note that, if $\frac{1}{p}+\frac{1}{q}=1$, then for  $\ell\in\Z$, we have
\begin{equation}
	\label{eq:lpnormprod}
\norm{\jc{n+\ell}}_{L^p(0,\infty)}\norm{\jc{n}}_{L^q(0,\infty)} \sim
\begin{cases}
	n^{\frac{2}{3p}-\frac{1}{2}}&\mbox{when }1<p<\frac{4}{3}\\
	\log n&\mbox{when }p=\frac{4}{3}\mbox{ or }p=4\\
	1&\mbox{when }\frac{4}{3}<p<4\\
	n^{\frac{1}{6}-\frac{2}{3p}}&\mbox{when }p>4
\end{cases}=o(n^{\frac{1}{6}})
\end{equation}
where the constants depend on $\ell$, for more details see \cite{BC}. We can now see that condition (L) is satisfied.

The expansion of  $\psi_{n,c}^{(\alpha)}$ in the basis of the spherical Bessel functions is done as follows.
First, by using \eqref{eq:sphbess}, we  calculate the scalar product
\begin{eqnarray*}
\scal{\psi_{n,c}^{(\alpha)},\jc{n}}_{L^2[0,+\infty[}
&=&\int_{0}^{+\infty}\sqrt{c}|\mu_{n,\alpha}(c)|\psi_{n,c}^{\alpha}(x)\jc{n}(x)\d x\\
&=&\sqrt{c}|\mu_{n,\alpha}(c)|\int_{0}^{+\infty}\psi_{n,c}^{\alpha}(x)\sqrt{2(2n+\alpha+1)}\frac{J_{2n+\alpha+1}(cx)}{\sqrt{cx}}\d x.
\end{eqnarray*}
Writing $\dst \nu=\sqrt{c}\frac{|\mu_{n,\alpha}(c)|}{\mu_{n,\alpha}(c)}\sqrt{2(2n+\alpha+1)}$
and since
$$
\dst\psi_{n,c}^{\alpha}(x)=\frac{1}{\mu_{n,\alpha}(c)}\hh_c^\alpha\psi_{n,c}^{\alpha}(x)
=\frac{1}{\mu_{n,\alpha}(c)}\int_0^1
\sqrt{cxy}J_{\alpha}(cxy) \psi_{n,c}^{\alpha}(y)\d y,
$$
then Fubini's theorem together with \eqref{eq:besseljacobi}, we get
\begin{eqnarray*}
\scal{\psi_{n,c}^{(\alpha)},\jc{n}}_{L^2[0,+\infty[}
&=&\nu\int_{0}^{1}\sqrt{y}\psi_{n,c}^{\alpha}(y)
\int_{0}^{+\infty}J_{2n+\alpha+1}(cx)J_{\alpha}(cxy)\d x\d y\\
&=&\frac{\nu}{c}\int_{0}^{1}y^{\alpha+\frac{1}{2}}P_n^{(\alpha,0)}(1-2y^2)\psi_{n,c}^{\alpha}(y)\d y.
\end{eqnarray*}
We thus have
$$
\scal{\psi_{n,c}^{(\alpha)},\jc{n}}_{L^2[0,+\infty[}=
(-1)^k\frac{|\mu_{n,\alpha}(c)|}{\sqrt{c}\mu_{n,\alpha}(c)}\scal{\vp,T_{n,\alpha}}_{L^2[0,1]}
$$
where $T_{n,\alpha}$ has been defined in \eqref{eq:tk}.
Writing $d_k^n=\scal{\vp,T_{k,\alpha}}_{L^2[0,1]}$, we thus get the following expansion on $[0,+\infty)$:
$$
\psi_{n,c}^{(\alpha)}(x)=\frac{|\mu_{n,\alpha}(c)|}{\sqrt{c}\mu_{n,\alpha}(c)}\sum_{k\ge 0}{(-1)^kd_k^n\jc{k}(x)}.
$$
Consider the differential operator given by
$$
\mathcal D_c^{\alpha}(\phi)(x)=-\frac{\d}{\d x}\ent{(1-x^2)\frac{\d}{\d x}}\phi(x)-\left(\frac{\frac{1}{4}-\alpha^2}{x^2}-c^2x^2\right)\phi(x).
$$
We know from \cite{Slepian3} that the operators $\mathcal D_c^{\alpha}$ and $\mathcal H_c^{\alpha}$ commute so that $\vp$ are eigenvectors of both
operators and we denote by $\chi_{n,\alpha}(c)$ the corresponding eigenvalue of $\dd_c^\alpha$, that is  
\begin{equation}
\label{eq:eigcircprol}
\dd_c^\alpha(\vp)=\chi_{n,\alpha}(c)\vp.
\end{equation}
Further more, we have the following inequality see \cite{Slepian3}:
\begin{equation}
\label{eq:slepeigest}
\left(\alpha+2n+\frac{1}{2}\right)\left(\alpha+2n+\frac{3}{2}\right)\le\chi_{n,\alpha}(c)
\le \left(\alpha+2n+\frac{1}{2}\right)\left(\alpha+2n+\frac{3}{2}\right)+c^2.
\end{equation}
According to \cite{Watson}, the Spherical Bessel functions are the eigenfunctions of the differential operator given by
\begin{equation*}
\mathcal{D}_c(\omega)=\left(\frac{x}{c}\right)^2\frac{d^2}{d^2x}(\omega)+\frac{(1+c)}{\sqrt{c}}\left(\frac{x}{c}\right)
\frac{d}{dx}(\omega)+c^2x^2\omega
\end{equation*}
and the corresponding eigenvalues are given by $\Big((2n+\alpha+1)^2-\frac{c^2-2c}{2c^3}\Big)\sqrt{2(2n+\alpha+1)}$.
It follows from Lemma \ref{lem:condD} that Condition $(D)$ of Theorem \ref{th:main}
is also satisfied.

If we substitute the expression of $\vp$ as a series of
Jacobi polynomials into \eqref{eq:eigcircprol}, we obtain the relations satisfied by the
coefficients $d_k^n$. More precisely, from \cite{Slepian3}, we obtain the three term recurrence relation
\begin{equation}
f(k,n,c,\alpha)d_k^n=a_{k,\alpha}d_{k-1}^n+a_{k+1,\alpha}d_{k+1}^n,~ \forall k\ge 0
\end{equation}
where $d_{-1}^{n}=0$ and
\begin{eqnarray}
\label{coeff-CPSWFs}
f(k,n,c,\alpha)&=&\displaystyle\frac{\chi_{n,\alpha}(c)-(\alpha+2k+\frac{1}{2})(\alpha+2k+\frac{3}{2})-c^2b_{k,\alpha}}{c^2} \\
a_{k,\alpha}&=&\frac{k(k+\alpha)}{(\alpha+2k)\sqrt{\alpha+2k+1}\sqrt{\alpha+2k-1}} \nonumber \\
b_{k,\alpha}&=&\frac{1}{2}\ent{\frac{\alpha^2}{(\alpha+2k+1)(\alpha+2k)}+1} \nonumber.
\end{eqnarray}

\subsubsection{The projection on the span of spherical Bessel functions}
Let $1<p<\infty$ and $\alpha\ge -\frac{1}{2}$.
For $n\geq 0$, let
$$
P_n^{(\alpha)}(x,y):=\sum_{k=0}^n\jc{k}(x)\jc{k}(y)
$$
and
$\pp_n^{(\alpha)}$ be the operator with kernel $P_n^{(\alpha)}(x,y)$. That is, $\pp_n^{(\alpha)}$ is the projection on
the span of $\{\jc{0},\ldots,\jc{n}\}$.

\begin{proposition}
\label{prop:projbessel}
Let $1<p<\infty$, $\alpha>-1/2$. Then the following estimate holds for
every $n$ and every $f\in L^p(0,\infty)$
$$
\norm{\pp_n^{(\alpha)}(f)}_{L^p(0,\infty)}\lesssim \begin{cases}\norm{f}_{L^p(0,\infty)}&\mbox{if }\dst\frac{4}{3}<p<4\\
n^{3/4}\norm{f}_{L^p(0,\infty)}&\mbox{otherwise}
\end{cases}
$$
with the implied constant independent of $f$ and $n$. 
\end{proposition}

\begin{proof}
The projection on the span of spherical Bessel functions has been studied by Varona \cite{J.L.VARONA} with a different normalization.
He considered
$$
j_n^{\alpha}(x)=\sqrt{2n+\alpha+1}J_{2n+\alpha+1}(\sqrt{x})x^{-\alpha/2-1/2}
$$
so that 
$$
\jc{n}=\sqrt{2}j_n^{\alpha}(c^2x^2)(cx)^{\alpha+1/2}.
$$

Next, if we define
$$
K_n(x,y)=\sum_{k=0}^nj_n^{\alpha}(x)j_n^{\alpha}(y)
$$ 
(with Varona's notation) then
$P_n^{(\alpha)}(x,y)
=2c^{2\alpha+1}(xy)^{\alpha+1/2}K_n(c^2x^2,c^2y^2)$.

Using Varona's computation \cite[page 69]{J.L.VARONA} we get

\begin{eqnarray}
P_n^{(\alpha)}(x,y)&=&\frac{(xy)^{1/2}}{x^2-y^2}\{xJ_{\alpha+1}(cx)J_{\alpha}(cy)-yJ_{\alpha+1}(cy)J_{\alpha}(cx)\}\nonumber\\
        &+&\frac{(xy)^{1/2}}{x^2-y^2}\{xJ^\prime_{\alpha+2n+2}(cx)J_{\alpha+2n+2}(cy)-yJ_{\alpha+2n+2}(cx)J^\prime_{\alpha+2n+2}(cy)\}.
				\label{eq:decomppn}
\end{eqnarray}

Now, recalling that $\hh$ denotes the Hilbert transform, it follows from \eqref{eq:decomppn} that
$$
\mathcal{P}_n^{(\alpha)}(f)(x)=\Omega_1(f)(x)-\Omega_2(f)(x)+\Omega_{3}(f)(x)-\Omega_{4}(f)(x)
$$
where
\begin{eqnarray*}
\Omega_1(f)(x)&=&\int_0^{\infty}\frac{(xy)^{1/2}}{x^2-y^2}xJ_{\alpha+1}(cx)J_{\alpha}(cy)f(y)\d y\\
	&=&\frac{x^{\frac{3}{2}}}{2}J_{\alpha+1}(cx)\mathcal{H}\big[y^{-1/4}J_{\alpha}(cy^{1/2})f(y^{1/2})\big](x^2).\\
\Omega_2(f)(x)&=&\int_0^{\infty}\frac{(xy)^{1/2}}{x^2-y^2}yJ_{\alpha}(cx)J_{\alpha+1}(cy)f(y)\d y\\
	&=&\frac{x^{\frac{1}{2}}}{2}J_{\alpha}(cx)\mathcal{H}\big[y^{1/4}J_{\alpha+1}(cy^{1/2})f(y^{1/2})\big](x^2).\\
\Omega_{3}(f)(x)&=&\int_0^{\infty}\frac{(xy)^{1/2}}{x^2-y^2}xJ'_{\alpha+2n+2}(cx)J_{\alpha+2n+2}(cy)f(y)\d y\\
	&=&\frac{x^{\frac{3}{2}}}{2}J'_{\alpha+2n+2}(cx)\mathcal{H}\big[y^{-1/4}J_{\alpha+2n+2}(cy^{1/2})f(y^{1/2})\big](x^2).\\
\Omega_{4}(f)(x)&=&\int_0^{\infty}\frac{(xy)^{1/2}}{x^2-y^2}yJ_{\alpha+2n+2}(cx)J'_{\alpha+2n+2}(cy)f(y)\d y\\
	&=&\frac{x^{\frac{1}{2}}}{2}J_{\alpha+2n+2}(cx)\mathcal{H}\big[y^{1/4}J'_{\alpha+2n+2}(cy^{1/2})f(y^{1/2})\big](x^2).
\end{eqnarray*}
Note that each of these operators is of the form
$$
\Omega_j(f)(x)=G_j(x)\hh\big[\ffi_j\big](x^2)
$$
so that
\begin{eqnarray*}
\norm{\Omega_j(f)}_{L_p(0,\infty)}^p&=&\int_0^\infty |G_j(x)|^p|\hh\big[\ffi_j\big](x^2)|^p\d x\\
&=&\int_0^\infty \frac{|G_j(\sqrt{x})|^p}{2\sqrt{x}}|\hh\big[\ffi_j\big](x)|^p\d x.
\end{eqnarray*}
But then, if we are able to find an upper bound $\omega_j\in A^p$ (see \cite{J.L.VARONA})  of $\dst \frac{|G_j(\sqrt{x})|^p}{2\sqrt{x}}\lesssim \omega_j(x)$, we obtain
$$
\norm{\Omega_j(f)}_{L_p(0,\infty)}^p\lesssim \norm{\hh\big[\ffi_j\big]}_{L_p\bigl((0,\infty),\omega_j(x)\d x\bigr)}^p
\lesssim \ent{\omega_j}_{A^p}^{\max(p,p')} \norm{\ffi_j}_{L_p\bigl((0,\infty),\omega_j(x)\d x\bigr)}^p.
$$
It then remains to prove that $\norm{\ffi_j}_{L_p\bigl((0,\infty),w_j(x)\d x\bigr)}^p\lesssim\norm{f}_{L_p(0,\infty)}^p$.

\smallskip

For $\Omega_1$,  $\ffi_1(y)=y^{-1/4}J_{\alpha}(cy^{1/2})f(x^{1/2})$. Further, we use the bound
$|J_{\alpha+1}(t)|\leq C_\alpha t^{-1/2}$ which allows us to chose $\omega_1(y)=y^{\frac{p-1}{2}}\in A^p$ since
$-1<\dst\frac{p-1}{2}<p-1$. Further
\begin{eqnarray*}
\norm{\ffi_1}_{L^p(\R_+,\omega_1(y)\d y)}^p&=&
\norm{y^{-1/4}J_{\alpha}(cy^{1/2})f(y^{1/2})}_{L^p(\R_+,\omega_1(y)\d y)}^p
\lesssim \int_0^\infty \abs{\frac{f(y^{1/2})}{(cy)^{1/2}}}^px^{\frac{p-1}{2}}\d y\\
&\lesssim&\norm{f}_{L^p(0,+\infty)}^p.
\end{eqnarray*}

\smallskip

We will now take care of $\Omega_2$. In this case
with $\ffi_2(y)=y^{1/4}J_{\alpha+1}(cy^{1/2})f(y^{1/2})$ and the same bound on the Bessel function shows that
we can chose $\omega_2(x)=x^{-\frac{1}{2}}\in A^p$.
\begin{eqnarray*}
\norm{\ffi_2}_{L^p(\R_+,\omega_2(x)\d x)}^p&=&
\norm{x^{1/4}J_{\alpha+1}(cx^{1/2})f(x^{1/2})}_{L^p(\R_+,\omega_2(x)\d x)}^p
\lesssim \int_0^\infty \abs{f(x^{1/2})}^px^{-\frac{1}{2}}\d x\\
&\lesssim&\norm{f}_{L^p(0,+\infty)}^p.
\end{eqnarray*}

\smallskip

The same reasoning would apply to $\Omega_{3},\Omega_{4}$ but with a bound that depends on $n$.
We thus need a more refined estimate which follows from \cite{BC}:
\begin{eqnarray*}
|J_{\mu}(x)|&\lesssim& x^{-\frac{1}{4}}\big(|x-\mu|+\mu^{\frac{1}{3}}\big)^{-\frac{1}{4}}\\
|J'_{\mu}(x)|&\lesssim& x^{-\frac{3}{4}}\big(|x-\mu|+\mu^{\frac{1}{3}}\big)^{\frac{1}{4}}.
\end{eqnarray*}
Set $\mu=\alpha+2n+2$. We may then take
\begin{equation*}
\omega_3(x)=x^{\frac{3p}{8}-\frac{1}{2}}\big(|c\sqrt{x}-\mu|+\mu^{\frac{1}{3}}\big)^{\frac{p}{4}}\quad\mbox{and}\quad
\omega_4(x)=x^{\frac{p}{8}-\frac{1}{2}}\big(|c\sqrt{x}-\mu|+\mu^{\frac{1}{3}}\big)^{-\frac{p}{4}}
\end{equation*}
By the Lemma \eqref{lem:ap}, $\omega_3$ and $\omega_4$ $\in A_p$ with $\ent{\omega_j}_{A^p}\lesssim 1$ if
$\dst\frac{4}{3}<p<4$ and  $\ent{\omega_j}_{A^p}\lesssim\mu^{3/4}$ otherwise.

Finally
$$
\ffi_3(x)=x^{-1/4}J_{\alpha+2n+2}(cx^{1/2})f(x^{1/2})
\quad\mbox{and}\quad
\ffi_4(x)=x^{1/4}J'_{\alpha+2n+2}(cx^{1/2})f(x^{1/2}).
$$
Note that
$$
|\ffi_3(x)|\lesssim x^{-3/8}\big(|c\sqrt{x}-\mu|+\mu^{\frac{1}{3}}\big)^{-\frac{1}{4}}|f(x^{1/2})|
\quad\mbox{and}\quad
|\ffi_4(x)|\lesssim x^{-1/8}\big(|c\sqrt{x}-\mu|+\mu^{\frac{1}{3}}\big)^{\frac{1}{4}}|f(x^{1/2})|.
$$
so that
$$
\norm{\ffi_j}_{L^p(\R_+,\omega_j(x)\d x)}^p
\lesssim \int_0^\infty x^{-\frac{1}{2}}|f(x^{1/2})|\d x\lesssim \norm{f}_{L^p(0,+\infty)}^p.
$$
It follows that $\norm{\Omega_j(f)}_{L^p(0,\infty)}\lesssim 1$ if $\dst\frac{4}{3}<p<4$ and
$\norm{\Omega_j(f)}_{L^p(0,\infty)}\lesssim n^{3/4}\norm{f}_{L^p(0,\infty)}$ for $1<p\leq \dst\frac{4}{3}$.
Grouping all estimates, the same holds for $\pp_n^{(\alpha)}$. Finally, as $\pp_n^{(\alpha)}$ is self-adjoint,
we also get the estimate $\norm{\pp_n^{(\alpha)}(f)}_{L^p(0,\infty)}\lesssim n^{3/4}\norm{f}_{L^p(0,\infty)}$ 
for $p\geq 4$.

\end{proof}

\subsection{Condition $(R)$}

We will now show that conditions $(R)$ of Theorem \ref{re:thmain} are satisfied, this is done in three lemmas.

\begin{lemma}
For every $k\geq 1$ and $\alpha>\dst-\frac{1}{2}$, $0\leq a_k^{(\alpha)}\leq \dst\frac{1}{2}$.
\end{lemma}

\begin{proof}
For
$k=1$, $a_{1,\alpha}=\frac{\sqrt{1+\alpha}}{(2+\alpha)\sqrt{\alpha+3}}$ which is clearly $\leq 1/2$ when $\alpha\geq0$. It is easy to see
that $a_{1,\alpha}$ is increasing with
$\alpha\in(-1/2,\alpha_0)$ and decreasing with $\alpha\in(\alpha_0,0)$ where $\alpha_0=\frac{-3+\sqrt{5}}{2}$.
Finally, $a_{1,\alpha_0}\sim0.3$ so that $a_{1,\alpha}\leq 1/2$ for every $\alpha$.
Write
$$
|a_{k,\alpha}|
=\frac{1}{4}\frac{1+\frac{2\alpha}{2k}}{\left(1+\frac{\alpha}{2k}\right)\sqrt{1+\frac{\alpha+1}{2k}}\sqrt{1+\frac{\alpha-1}{2k}}}
=\frac{1}{4}\psi(1/2k),
$$
where $\dst\psi(x)=\frac{1+2\alpha x}{(1+\alpha x)\big(1+(\alpha-1)x\big)^{1/2}\big(1+(\alpha+1)x\big)^{1/2}}$. It is thus enough to show that
$\abs{\psi(x)}\leq 2$ for $x\in[0,1/4]$. Note that $\psi$ is non-negative
for $\alpha>-1/2$ and $x\leq 1$.
When $-1/2<\alpha\leq 0$, as $1+2\alpha x\leq 1+\alpha x$
and $1+(\alpha+1)x\geq 1$,
$$
\psi(x)\leq\frac{1}{\sqrt{1+(\alpha-1)x}}
\leq \frac{1}{\sqrt{1-3x/2}}\leq\frac{1}{\sqrt{5/8}}<2
$$
when $x\leq1/4$. 	When $\alpha> 0$, we first bound
$$
\psi(x)\leq\frac{1+2\alpha x}{(1+\alpha x)\big(1+(\alpha-1)x\big)}
$$
and it is enough to prove that, for $x>0$, we have
$$
1+2\alpha x\leq 2(1+\alpha x)\big(1+(\alpha-1)x\big)
=2+(4\alpha-2)x+2\alpha(\alpha-1)x^2.
$$
or, equivalently, $1+2(\alpha-1)x+2\alpha(\alpha-1)x^2\geq 0$.
When $\alpha\geq1$ this is obvious, while for $0<\alpha<1$, the roots of this equation are
$$
-\frac{1-\alpha+\sqrt{1-\alpha^2}}{2\alpha(1-\alpha)}<0<\frac{\sqrt{1-\alpha^2}-(1-\alpha)}{2\alpha(1-\alpha)}
=\frac{\sqrt{1+\alpha}-\sqrt{1-\alpha}}{2\alpha\sqrt{1-\alpha}}.
$$
The inequality is therefore satisfied as soon as $x\leq 1/2k$ with
$$
k\geq k_\alpha
:=\frac{\alpha\sqrt{1-\alpha}}{\sqrt{1+\alpha}-\sqrt{1-\alpha}}
=\frac{\alpha(1-\alpha+\sqrt{1-\alpha^2})}{2}.
$$
As $0<\alpha<1$, it is easy to see that $k_\alpha\leq 1$.
\end{proof}

Let us now estimate the $b_{k,\alpha}$'s:

\begin{lemma}
	\label{lem:2}
	For every $k$ and every $\alpha>-1/2$,
	$b_{k,\alpha}=\frac{1}{2}+\tilde\eta_{k,\alpha}$
	with $|\tilde\eta_{k,\alpha}|\leq \frac{1}{2}$.	
\end{lemma}

\begin{proof}
	From the definition of $b_{k,\alpha}$,
	$\tilde\eta_{k,\alpha}=\frac{\alpha^2}{2(\alpha+2k+1)(\alpha+2k)}$.	When $\alpha= 0$, $\eta=0$ and when $\alpha>0$ we directly get
	$\tilde\eta_{k,\alpha}\leq \frac{1}{2}$.
	When $-1/2<\alpha<0$, $\alpha+j>1/2>|\alpha|$ for every $j\geq 1$ thus
	$|\tilde\eta_{0,\alpha}|=\frac{|\alpha|}{2(\alpha+1)}\leq\frac{1}{2}$
	while for $k\geq 1$
	we directly get $0\leq\tilde\eta_{k,\alpha}=\frac{1}{2}\frac{|\alpha|}{\alpha+2k}\frac{|\alpha|}{\alpha+2k+1}\leq\frac{1}{2}$.
\end{proof}

The last step consists in establishing the bounds
for $\abs{f(k,n,c,\alpha)}$. But from \eqref{eq:slepeigest} and \eqref{coeff-CPSWFs}, it is straightforward to see that
$f(k,n,c,\alpha)$ satisfies the conditions of Lemma \ref{lem:fkn}. In summary

\begin{lemma}
	\label{lem:3}
	For every $\alpha>-1/2$, every $c>0$,
	
	--- for fixed $n$, $f(k,n,c,\alpha)\gtrsim k^2$, when $k$ is large enough;
	
	--- for every $n\geq c^2/2$, $k\geq0$, $k\not=n$, we have
	$$
	\abs{f(k,n,c,\alpha)}\geq 4\frac{|k-n|k+c^2}{c^2};
	$$
	
	--- for every $n\geq c^2/2$,
	$$
	\abs{\frac{a_{n+1}^{(\alpha)}}{f(n+1,n,c,\alpha)}-\frac{a_{n+2}^{(\alpha)}}{f(n+2,n+1,c,\alpha)}}\lesssim n^{-2}.
	$$
	
\end{lemma}

\subsection{Condition $(B')$}
It remains to check condition $(B')$ that is, to estimate the $L^p$ norm of the operator with kernel
$$
Q_N^{(\alpha)}(x,y)=\sum_{n=n_0}^N\big(\jc{n}(x)\jc{n+1}(y)+\jc{n+1}(x)\jc{n}(y)\big).
$$
\begin{lemma}
\label{lem:BC}
Let $1<p<\infty$ then, for every $f\in L^p(0,\infty)$
$$
\left(\int_0^{\infty}\abs{\int_0^{\infty}Q_N^{(\alpha)}(x,y)f(y)\d y}^p\d x\right)^{1/p}\lesssim N^{2/3}\norm{f}_{L^p(0,\infty)}.
$$
and the implied constant is independent of $N$ and $f$.
\end{lemma}

\begin{proof}
First, using the identity (see \cite{Watson}) 
\begin{equation}
\label{eq:watson}
\frac{2\nu}{x}J_{\nu}(x)=J_{\nu+1}(x)+J_{\nu-1}(x)
\end{equation}
twice, one gets
\begin{eqnarray*}
J_{2n+\alpha+3}(x)&=&\frac{4(2n+\alpha+1)(2n+\alpha+2)}{x^2}J_{2n+\alpha+1}(x)\\
&&-\frac{2(2n+\alpha+2)}{x}J_{2n+\alpha}(x)-J_{2n+\alpha+1}(x).
\end{eqnarray*}
so that
\begin{eqnarray*}
\lefteqn{J_{2n+\alpha+1}(x)J_{2n+\alpha+3}(y)+J_{2n+\alpha+1}(y)J_{2n+\alpha+3}(x)} \\
 &&\qquad\qquad=4(2n+\alpha+1)(2n+\alpha+2)\left(\frac{1}{x^2}+\frac{1}{y^2}\right)J_{2n+\alpha+1}(x)J_{2n+\alpha+1}(y) \\
&&\qquad\qquad\qquad
-2(2n+\alpha+2)\left(\frac{1}{y}J_{2n+\alpha}(y)J_{2n+\alpha+1}(x)+\frac{1}{x}J_{2n+\alpha}(x)J_{2n+\alpha+1}(y)\right) \\
&&\qquad\qquad\qquad -2J_{2n+\alpha+1}(x)J_{2n+\alpha+1}(y).
\end{eqnarray*}
Using again \eqref{eq:watson} for the middle term, we get
\begin{eqnarray*}
\lefteqn{J_{2n+\alpha+1}(x)J_{2n+\alpha+3}(y)+J_{2n+\alpha+1}(y)J_{2n+\alpha+3}(x)} \\
&&\qquad\qquad =4(2n+\alpha+1)(2n+\alpha+2)\left(\frac{1}{x^2}+\frac{1}{y^2}\right)J_{2n+\alpha+1}(x)J_{2n+\alpha+1}(y) \\
&&\qquad\qquad\ -2\frac{(2n+\alpha+2)}{2n+\alpha}\Bigl(\bigl(J_{2k+\alpha+1}(y)+J_{2k+\alpha-1}(y)\bigr)J_{2n+\alpha+1}(x)\\
&&\qquad\qquad\ +\bigl(J_{2n+\alpha+1}(x)+J_{2n+\alpha-1}(x)\bigr)J_{2n+\alpha+1}(y)\Bigr) \\
&&\qquad\qquad\ -2J_{2n+\alpha+1}(x)J_{2n+\alpha+1}(y).
\end{eqnarray*}

Next, since $\jc{n}(x)=\sqrt{2(2n+\alpha+1)}\frac{J_{2n+\alpha+1}(cx)}{\sqrt{cx}}$,
one obtains
\begin{eqnarray}
\lefteqn{\jc{n}(x)\jc{n+1}(y)+\jc{n}(y)\jc{n+1}(x)}\nonumber \\
 &&\qquad\qquad= \frac{4}{c^2}\sqrt{(2n+\alpha+1)(2n+\alpha+3)}(2n+\alpha+2)\left(\frac{1}{x^2}+\frac{1}{y^2}\right)
\jc{n}(x)\jc{n}(y) \nonumber\\
&&\qquad\qquad\ -2\frac{(2n+\alpha+2)}{2n+\alpha}\sqrt{\frac{2n+\alpha+3}{2n+\alpha-1}}\Bigl(\jc{n-1}(y)\jc{n}(x)+\jc{n-1}(x)\jc{n}(y)\Bigr) \nonumber\\
&&\qquad\qquad\ -2\sqrt{\frac{2n+\alpha+3}{2n+\alpha+1}}\left(\frac{2(2n+\alpha+2)}{2n+\alpha}+1\right)\jc{n}(x)\jc{n}(y).
\label{eq:jcnjcn+1}
\end{eqnarray}
Now we write
$$
\gamma_n=\frac{4}{c^2}\sqrt{(2n+\alpha+1)(2n+\alpha+3)}(2n+\alpha+2),
$$
so that $|\gamma_n|\lesssim n^2$,
$$
2\frac{(2n+\alpha+2)}{2n+\alpha}\sqrt{\frac{2n+\alpha+3}{2n+\alpha-1}}=2+\kappa_n,\qquad 0\leq\kappa_n\lesssim n^{-1}
$$
and 
$$
2\sqrt{\frac{2n+\alpha+3}{2n+\alpha+1}}\left(\frac{2(2n+\alpha+2)}{2n+\alpha}+1\right)
=6+\tilde\kappa_n,\qquad 0\leq\tilde\kappa_n\lesssim n^{-1}.
$$
Then, summing \eqref{eq:jcnjcn+1} from $n=n_0$ to $n=N$ gives
\begin{eqnarray*}
Q_N^{(\alpha)}(x,y)&=&\left(\frac{1}{x^2}+\frac{1}{y^2}\right)\sum_{n=n_0}^N \gamma_n\jc{n}(x)\jc{n}(y)\\
&&-(2+\kappa_{n_0})\Bigl(\jc{n_0-1}(y)\jc{n_0}(x)+\jc{n_0-1}(x)\jc{n_0}(y)\Bigr)\\
&&-2Q_N^{(\alpha)}(x,y)-\sum_{n=n_0}^N\kappa_n\Bigl(\jc{n}(y)\jc{n+1}(x)+\jc{n}(x)\jc{n+1}(y)\Bigr)\\
&&+(2+\kappa_{N})\Bigl(\jc{N}(y)\jc{N+1}(x)+\jc{N}(x)\jc{N+1}(y)\Bigr)\\
&&-6P_N^{(\alpha)}(x,y)+6P_{n_0-1}(x,y)-\sum_{n=n_0}^N\tilde\kappa_n\jc{n}(x)\jc{n}(y).
\end{eqnarray*}
It follows that
$$
Q_N^{(\alpha)}(x,y)=Q_{N,1}(x,y)+\cdots+Q_{N,7}(x,y).
$$
Using \eqref{eq:lpnormprod} and Lemma \ref{lem:triv} we get that
$$
\norm{Q_{N,2}}_{L^p(0,\infty)\to L^p(0,\infty)},\norm{Q_{N,6}}_{L^p(0,\infty)\to L^p(0,\infty)}\lesssim 1
$$
and
$$
\norm{Q_{N,4}}_{L^p(0,\infty)\to L^p(0,\infty)}\lesssim N^{1/6}
$$
while
$$
\norm{Q_{N,3}}_{L^p(0,\infty)\to L^p(0,\infty)}\lesssim \sum_{n=n_0}^N\frac{1}{n}n^{1/6}\lesssim N^{1/6}
$$
and
$$
\norm{Q_{N,7}}_{L^p(0,\infty)\to L^p(0,\infty)}\lesssim \sum_{n=n_0}^N\frac{1}{n}n^{1/6}\lesssim N^{1/6}.
$$
We have seen in Proposition \ref{prop:projbessel} that $\norm{Q_{N,5}}_{L^p(0,\infty)\to L^p(0,\infty)}\lesssim N^{3/4}$.

Concerning $Q_{N,1}$, we will use the following equality, see \cite{BC},
$$
\norm{x^{-2}j_n^{\alpha}}_{L^p(0,\infty)}\norm{j_n^{\alpha}}_{L^q(0,\infty)}+\norm{j_n^{\alpha}}_{L^p(0,\infty)}
\norm{x^{-2}j_n^{\alpha}}_{L^q(0,\infty)}=O(n^{-\frac{7}{3}})
$$
from which we deduce that
$$
\norm{Q_{N,1}}_{L^p(0,\infty)\to L^p(0,\infty)}\lesssim \sum_{n=n_0}^Nn^2n^{-7/3}\lesssim N^{2/3}.
$$
By grouping all estimates, we obtain $\norm{Q_{N}^{(\alpha)}}_{L^p(0,\infty)\to L^p(0,\infty)}\lesssim N^{2/3}$ as claimed.
\end{proof}

\subsection{Conclusion}

It remains to conclude. ll conditions of Theorem \ref{th:main} are satisfied. Therefore, the Hankel prolate spheroidal series
converges in $L^p(0,\infty)$ if and only if the Bessel series converge. The later ones converge in $L^p(0,\infty)$
if and only if $p\in(4/3,4)$. We have thus proved the following:

\begin{theorem}
Let $\alpha>-1/2$ and $c>0$, $N\geq 0$.

Let $(\psi_{n,c}^{(\alpha)})_{n\geq 0}$ be the family of circular prolate spheroidal wave functions.
For a smooth function $f$ on $I=(0,\infty)$, define
$$
\Psi^{(\alpha)}_Nf=\sum_{n=0}^N\scal{f,\psi_{n,c}^{(\alpha)}}_{L^2(0,\infty)}\psi_{n,c}^{(\alpha)}.
$$
Then, for every $p\in(1,\infty)$, $\Psi^{(\alpha)}_N$ extends to a bounded operator $L^p(0,\infty)\to L^p(0,\infty)$.
Further
$$
\Psi^{(\alpha)}_Nf\to f\qquad \mbox{in }L^p(0,\infty)
$$
for every $f\in B_{c,p}^{\alpha}$ if and only if $p\in(4/3,4)$.
\end{theorem}

\end{document}